\documentclass[a4paper]{article}
\usepackage{geometry}
\usepackage{authblk}
\usepackage{amsthm}
\usepackage{amssymb}
\usepackage{amsmath}
\usepackage{xspace}
\usepackage{fancyvrb}
\usepackage{makecell}
\usepackage{mathtools}
\usepackage{hyperref}
\usepackage{cleveref}
\usepackage{thm-restate}

\usepackage[only, llbracket, rrbracket, llparenthesis,
            rrparenthesis, rightarrowtriangle]{stmaryrd}
\usepackage{mathpartir}
\usepackage{tikz-cd}
\usepackage{quiver}
\usetikzlibrary{nfold}

\theoremstyle{plain}
\newtheorem{theorem}{Theorem}[section]
\newtheorem{corollary}[theorem]{Corollary}
\newtheorem{lemma}[theorem]{Lemma}
\newtheorem{proposition}[theorem]{Proposition}

\theoremstyle{definition}
\newtheorem{definition}[theorem]{Definition}

\newtheorem{example}[theorem]{Example}

\theoremstyle{remark}

\newtheorem{claim}[theorem]{Claim}

\usepackage{macros}
\allowdisplaybreaks

\title{A type theory for invertibility \texorpdfstring{\\}{} in weak
  \texorpdfstring{\(\omega\)}{ω}-categories}

\author[1]{Thibaut Benjamin}
\affil[1]{Université Paris-Saclay, CNRS, ENS Paris-Saclay, LMF, France}
\author[2]{Camil Champin}
\affil[2]{ENS de Lyon, France}
\author[3]{Ioannis Markakis}
\affil[3]{University of Cambridge, UK}
\affil[ ]{}
{
    \makeatletter
    \renewcommand\AB@affilsepx{: \protect\Affilfont}
    \makeatother
    \affil[ ]{Emails}

    \makeatletter
    \renewcommand\AB@affilsepx{, \protect\Affilfont}
    \makeatother

    \affil[1]{thibaut.benjamin@universite-paris-saclay.fr}
    \affil[2]{camil.champin@ens-lyon.fr}
    \affil[3]{ioannis.markakis@cl.cam.ac.uk}
}
\date{}
  
\begin{document}

\maketitle

\begin{abstract}
  We present a conservative extension \cattinv of the dependent type theory
  \catt for weak \(\omega\)\-categories with a type witnessing coinductive
  invertibility of cells. This extension allows for a concise description of the
  ``walking equivalence'' as a context, and of a set of maps characterising
  \(\omega\)\-equifibrations as substitutions. We provide an implementation of
  our theory, which we use to formalise basic properties of invertible cells.
  These properties allow us to give semantics of \cattinv in marked weak
  \(\omega\)\-categories, building a fibrant marked \(\omega\)\-category out of
  every model of \cattinv.
\end{abstract}

\section{Introduction}

Homotopy type theory endows every type with the structure of a weak higher
groupoid, arising from its tower of iterated identity
types~\cite{lumsdaineWeak$omega$CategoriesIntensional2009,
vandenbergTypesAreWeak2011, altenkirchSyntacticalApproachWeak2012}. Building on
this observation, Brunerie developed a type‑theoretic presentation of weak
\(\omega\)-groupoids~\cite{brunerie2016homotopy}, and
Finster and Mimram~\cite{finsterTypetheoreticalDefinitionWeak2017} introduced the directed variant
\catt, a type theory for weak \(\omega\)-categories. In this paper, we introduce
\cattinv, an extension of \catt internalising \emph{invertibility}.

The notion of invertibility in higher categories is in general subtle. In an
ordinary category, a morphism \(f\) is invertible if it has an inverse
\(f^{-1}\) such that \(f\circ f^{-1} = \id\) and \(f^{-1}\circ f = \id\). In a
bicategory, these equalities are replaced by the existence of invertible
\(2\)\-cells. Similarly for finite-dimensional categories, we define
invertibility inductively by replacing equalities with higher-dimensional
invertible cells. For \(\omega\)\-categories, which are infinite-dimensional,
this process is to be understood coinductively, so showing that a cell is
invertible amounts to producing an infinite amount of data. The proposed theory
\cattinv is a syntactic method for reasoning about such infinite structures.

Invertibility plays a crucial role in the homotopy theory of higher categories,
in particular, in defining weak equivalences between them: functors that are
surjective up to invertibility and weak equivalences on hom higher categories.
In the case of strict \(\omega\)\-categories, where the composition operations
are associative and unital on the nose, these weak equivalences are part of a
model structure~\cite{lafontFolkModelStructure2010}. The analogous statement for
weak \(\omega\)\-categories, or even \(\omega\)\-groupoids, still remains open,
while it has been proven for other models~\cite{loubaton2024categorical,
  chanavat2024model}. The case of groupoids has been solved up to dimension
\(3\)~\cite{henry2023homotopy} by constructing a semi-model structure equivalent
to homotopy \(3\)-types. The directed setting has been studied by Fujii, Hoshino
and Maehara~\cite{fujii$omega$weakEquivalencesWeak2024,
  fujii$o$equifibrationsStrictWeak2025} who show that weak equivalences satisfy
the 2-of-6 property and that \emph{equifibrations} are characterised by a right
lifting property.

By analogy with the groupoidal case~\cite{henryAlgebraicModelsHomotopy2016,
  lanariGlobularPathObject2020}, the missing ingredient for building a
semi-model structure for weak \(\omega\)-categories is the construction of
reversible cylinders over \catt contexts. Such a semi-model would provide an
\(\infty\)-category of \(\omega\)-categories, in particular, spaces of functors
together with the appropriate notion of (homotopy) limits and colimits of
\(\omega\)-categories. This is the first step towards comparing
\(\omega\)-categories with homotopical models of \((\infty,\infty)\)-categories.

In the context of the type theory \catt, there has been a partial construction
of directed (i.e.~not reversible)
cylinders~\cite{benjaminNaturalityHigherdimensionalPath2025}. This construction
proceeds by duplicating all the variables of the context, and then adjoining new
variables that relate the original variables and their copies. Reversible
cylinders can be obtained by making the relating variables invertible.
Therefore, the reversible cylinders over a \catt context can be described as
\cattinv contexts, but not \catt contexts themselves. For example, we expect the
cylinder over the point to be the walking equivalence \(\Equiv{1}\) of
Section~\ref{sec:cattinv}, which we present in this paper as a \cattinv context.

\subsection*{Contributions}
We propose an extension \cattinv of the type
theory \catt for weak \(\omega\)-categories, with a type \(\Inv(t)\) of
invertibility structures over a term \(t\). This type is equipped with
destructors that provide invertibility data for \(t\), as well as constructors
internalising closure properties of invertible cells. We also postulate the usual
\(\beta/\eta\)\-rules and show that the resulting theory is normalising for an
important class of terms.

We provide an implementation\footnote{available at
\url{https://zenodo.org/records/18343317}} for \cattinv and we use it to
formalise results about invertibility structures that are not expressible in
the theory \catt. We take this as evidence that this
theory is well suited to working in the presence of invertibility structures,
and may be a valuable asset for further investigating the construction of a
model structure on weak \(\omega\)-categories.

We show that \cattinv is a conservative extension of \catt, and that the
inclusion between their respective syntactic categories is fully faithful. This
inclusion allows us to induce an \(\omega\)\-category from a model of \cattinv,
providing semantics of \cattinv into \(\omega\)\-categories. Using these
semantics, we are able to reconstruct the walking equivalence of
Ozornova and Rovelli~\cite{ozornovaWhatEquivalenceHigher2024},
which is used in the characterisation
of equifibrations~\cite{fujii$o$equifibrationsStrictWeak2025}. Finally, we
generalise the definition of marked strict
\(\omega\)\-categories~\cite{loubatonInductiveModelStructure2023} to the weak
case, and we promote our semantics to land in fibrant marked
\(\omega\)\-categories.

\subsection*{Plan of the paper}
In \Cref{sec:catt}, we present the dependent
type theory \catt. In \Cref{sec:semantics-catt}, we present the notion of models
of a theory and we use it to define weak \(\omega\)-categories. We introduce the
theory \cattinv in \Cref{sec:cattinv}, before formalising some constructions of
invertibility in \Cref{sec:examples}. In \Cref{sec:cattinv-semantics}, we show
that \cattinv is conservative over \catt and prove that we can recover the
walking equivalence as a context in \cattinv.
Finally, in \Cref{sec:marked}, we introduce marked weak
\(\omega\)-categories and construct a semantics of \cattinv in fibrant marked
\(\omega\)\-categories.


\subsection*{Acknowledgements}
We would like to thank Jamie Vicary for his support during this project.

\section{The type theory \texorpdfstring{\catt}{CaTT}}
\label{sec:catt}

We start with a brief presentation of the type theory
\catt~\cite{finsterTypetheoreticalDefinitionWeak2017} and the suspension
operation on it~\cite{benjaminHom$omega$categoriesComputad2024}. The following
tables summarise the raw syntax of \catt and its judgement forms. Here,
\(\mathcal{V}\) is a countably infinite set of variables.
\begin{align*}
  \begin{tabular}{|lr@{\;\;}l@{\;}l|}
    \hline
    Variable & \(x\Coloneqq\)
      & \multicolumn{2}{l|}{\(x\in\mathcal{V}\)} \\
    Type & \(A,B\Coloneqq\)
      & \(\obj\) & \(\vert\ \arr[A]{t}{u}\) \\
    Term & \(t,u\Coloneqq\)
      & \(x\) & \(\vert\ \coh_{\Gamma,A}[\gamma]\)\\
    Context & \(\Gamma,\Delta\Coloneqq\)
      & \(\emptycontext\) & \(\vert\ \Gamma \extctx (x:A)\) \\
    Substitution & \(\gamma,\delta\Coloneqq\)
      & \(\emptysub\) & \(\vert\ \gamma \extsub (x \mapsto t)\) \\
    \hline
  \end{tabular} &&
  \begin{tabular}{|l|l|}
    \hline
    \(\ctxty{\Gamma}\)
      & \text{\(\Gamma\) is a context}\\
    \(\typing{\Gamma}{A}\)
      & \text{\(A\) is a type in context \(\Gamma\)} \\
    \(\typing{\Gamma}{t}[A]\)
      & \text{\(t\) is a term of type \(A\) in \(\Gamma\)}\\
    \(\typing{\Gamma}{\gamma}[\Delta]\)
      & \text{\(\gamma\) is a substitution from \(\Delta\) to \(\Gamma\)}\\
    \(\psctxty{\Gamma}\)
      & \text{\(\Gamma\) is a pasting context}\\
    \(\pstyping{\Gamma}{x}[A]\)
      & \text{\(x\) is a dangling variable in \(\Gamma\)} \\
    \(\fulltyping{\Gamma}{A}\)
      & \text{\(A\) is a full type in context \(\Gamma\)} \\
    \hline
  \end{tabular}
\end{align*}
The judgements of \catt are defined mutually inductively using a set of rules.
We will call contexts, types, terms, and substitutions \emph{valid} when
the corresponding judgements can be derived. We require first the usual
structural rules for variables, contexts and substitutions, which guarantee the
existence of identity substitutions \(\typing{\Gamma}{\id_\Gamma}[\Gamma]\). The
type formation rules of \catt are the following ones:
\begin{mathpar}
  \inferdef[\(\obj\)in]{\ctxty{\Gamma}}{\typing{\Gamma}{\obj}}
  \label{rule:obj-intro}
  \and
  \inferdef[\(\to_0\)in]{
    \typing{\Gamma}{u}[\obj] \\
    \typing{\Gamma}{v}[\obj] }
  {\typing{\Gamma}{\arr[\obj]{u}{v}}}
  \label{rule:arr-intro1}
  \and
  \inferdef[\(\to\)in]{
    \typing{\Gamma}{u}[\arr[A]{v}{w}] \\
    \typing{\Gamma}{u'}[\arr[A]{v}{w}] }
  {\typing{\Gamma}{\arr[{\arr[A]{v}{w}}]{u}{u'}}}
  \label{rule:arr-intro2}
\end{mathpar}
For the sake of readability, we sometimes omit the index \(A\) in the expression
\(\arr[A]{u}{v}\), when it can be inferred. While the second and third rule
could be unified, we choose to keep them separate so as to prevent formation of
unwanted types down the line in the theory \cattinv. We note that we have used
named variables for the sake of readability, but we could have instead presented
the theory using De Bruijn indices. The judgements \(\psctxty{\Gamma}\)
parameterise a set of contexts that correspond to \emph{pasting diagrams}, the
arities of the operations of globular \(\omega\)\-categories. The rules for this
judgement are given below:
\begin{mathpar}
  \inferdef[pss]
  { }
  {\pstyping{(x\ty\obj)}{x}[\obj]}
  \label{rule:pss}
  \and
  \inferdef[psd]
  {\pstyping{\Gamma}{f}[\arr[A]{x}{y}]}
  {\pstyping{\Gamma}{y}[A]}
  \label{rule:psd}
  \and
  \inferdef[pse]
  {\pstyping{\Gamma}{x}[A]}
  {\pstyping{\Gamma\extctx(y\ty A)\extctx(f\ty \arr[A]{x}{y})}{f}[{\arr[A]{x}{y}}]}
  \label{rule:pse}
  \and
  \inferdef[ps]
  {\pstyping{\Gamma}{x}[\obj]}
  {\psctxty{\Gamma}}
  \label{rule:ps}
\end{mathpar}
We define the \emph{dimension} of types to be the number of nested arrows
in them minus one and the dimension of terms to be the dimension of their type
plus one. The dimension of contexts is the maximum of the dimensions of its
variables. A \emph{source} (resp. \emph{target}) \emph{variable} of
\(\psctxty{\Gamma}\) is one that is not the target (resp. source) of another
variable. We write  \(\fulltyping{\Gamma}{\arr[A]{u}{v}}\) and say that
\(\arr[A]{u}{v}\) is a \emph{full type} of a pasting context \(\Gamma\) of
dimension \(n\) when either of the two following conditions is satisfied:
\begin{itemize}
  \item \(\dim u = n - 1\), and \(u\) and \(v\) use all source and target
    variables of \(\Gamma\) of dimension at most \(n-1\) respectively.
  \item both \(u\) and \(v\) use all variables of \(\Gamma\),
\end{itemize}
This definition of full type is equivalent to the original
one~\cite[Proposition~2.10]{deanComputadsWeak$omega$categories2024}. The
intuition behind this definition is the following: in the first case, \(u\) and
\(v\) can be thought of as ways to compose the source and target of the pasting
diagram \(\Gamma\), and the rule below gives a way to compose the entire diagram
\(\Gamma\). The second case is used to ensure a weak form of uniqueness, akin to
contractibility, of such composites: \(u\) and \(v\) can be thought of as ways
to compose \(\Gamma\) or as coherence data relating such composites. The term
introduction rule of the type theory \catt then takes the form:
\begin{mathpar}
  \inferdef[\(\coh\)-in]{
    \psctxty{\Gamma} \\
    \fulltyping{\Gamma}{A} \\
    \typing{\Delta}{\gamma}[\Gamma]
  }{
    \typing{\Delta}{\coh_{\Gamma,A}[\gamma]}[A[\gamma]]
  }
  \label{rule:coh-intro}
\end{mathpar}

\begin{example}\label{ex:globular-ctx}
  The following are valid contexts of the type theory \catt :
  \begin{align*}
    \Gamma &= (x:\obj) \extctx (y \ty \obj) \extctx
      (f \ty \arr{x}{y})\extctx  (g\ty\arr{x}{y}) \extctx
      (a : \arr{f}{g}) \extctx (h\ty \arr{x}{x}) \\
    \Delta &= (x\ty\obj) \extctx (y\ty\obj) \extctx (f\ty \arr{x}{y}) \extctx
      (g\ty\arr{x}{y}) \extctx (a\ty\arr{f}{g}) \extctx
      (z\ty\obj) \extctx (h\ty\arr{y}{z})
  \end{align*}
  We visualise contexts as follows by drawing variables of type \(\obj\) as
  points, and variables of type \(\arr{u}{v}\) as arrows from \(u\) to \(v\).
  \begin{align*}
    \begin{tikzcd}[ampersand replacement=\&]
      x
      \ar[r, bend left = 40,"f",""{below, name = A}]
      \ar[r, bend right = 40,"g"{below},""{name = B}]
      \ar[loop left, out=-150, in=150, distance=3em, "h"{left}]
      \& y \ar[from = A, to = B, Rightarrow, "a"]
    \end{tikzcd} &&
    \begin{tikzcd}[ampersand replacement=\&]
      x
      \ar[r, bend left = 40,"f",""{below, name = A}]
      \ar[r, bend right = 40,"g"{below},""{name = B}]
      \ar[from = A, to = B, Rightarrow, "a"]
      \& y
      \ar[r, "h"]
      \& z
    \end{tikzcd}
  \end{align*}
  Also \(\Delta\) is a pasting with source variables \(f\), \(a\) and
  \(h\), and with target variables \(g\), \(a\) and \(h\).
\end{example}

\begin{example}\label{ex:id-comp}
  Consider the following pasting contexts:
  \begin{align*}
    \Gamma_{\id} = (x \ty \obj) &&
    \Gamma_{\comp{}{}} = (x\ty\obj) \extctx (y\ty\obj) \extctx (f\ty \arr{x}{y})
      \extctx (z\ty\obj) \extctx (g\ty\arr{y}{z})
  \end{align*}
  By fullness of \(A_{\id} = \arr{x}{x}\) and
  \(A_{\comp{}{}} = \arr{x}{z}\), we derive the validity of the following terms:
  \begin{align*}
    \typing{\Gamma_{\id}}{\coh_{\Gamma_{\id},A_{\id}}[\id]}[A_{\id}] &&
    \typing{\Gamma_{\comp{}{}}}{\coh_{\Gamma_{\comp{}{},A_{\comp{}{}}}}[\id]}[A_{\comp{}{}}]
  \end{align*}
  which we denote by \(\id x\) and \(\comp{f}{g}\) respectively.
\end{example}

\begin{example}
  The walking triangle is the context of \catt given by:
  \begin{align*}
    \begin{array}{l}
      (x\ty\obj) \extctx (y\ty\obj) \extctx (f\ty \arr{x}{y}) \extctx
      (z\ty\obj) \extctx \\(g\ty\arr{y}{z})\extctx(h\ty \arr{x}{z})\extctx
      (\alpha \ty \arr{\comp{f}{g}}{h})
    \end{array}
    &&
    \begin{tikzcd}[ampersand replacement = \&, row sep = tiny]
      \& y \\
      x \&\& z
      \arrow["g", from=1-2, to=2-3]
      \arrow["f", from=2-1, to=1-2]
      \arrow[""{name=0, anchor=center, inner sep=0}, "h"', from=2-1, to=2-3]
      \arrow[shorten <= 0.2em, shorten >= 0.1em, Rightarrow, from=1-2, to=0]
    \end{tikzcd}
  \end{align*}
\end{example}

\noindent
\emph{Suspension} is a meta-operation preserving all judgments of \catt, which
corresponds semantically to the formation of hom
\(\omega\)\-categories~\cite{benjaminHom$omega$categoriesComputad2024}. Homotopy
type theorists may
recognise it as the fact that every construction on identity types may be
interpreted in higher identity types by changing the base type. The
suspension is defined on valid contexts, types, terms and
substitutions by the following mutually recursive formulas, where
the variables \(v^-\) and \(v^+\) are assumed to be fresh in \(\Gamma\):
\begin{mathpar}
  \Susp (\emptycontext)
    = (v^-\ty \obj) \extctx (v^+\ty \obj) \and
  \Susp (\Gamma\extctx (x \colon A))
    = (\Susp \Gamma)\extctx (\Susp x \colon \Susp A) \and
  \Susp (\obj)
    = \arr[A]{v^-}{v^+} \and
  \Susp (\arr[A]{u}{v})
    = \arr[\Susp A]{\Susp u}{\Susp v} \and
  \Susp (x)
    = x \and
  \Susp (\coh_{\Gamma,A}[\gamma])
    =  \coh_{\Susp \Gamma,\Susp A}[\Susp \gamma] \and
  \Susp(\emptysub)
    = (v^-\mapsto v^-) \extsub (v^+\mapsto v^+) \and
  \Susp(\gamma\extsub (x\mapsto t))
    = (\Susp \gamma)\extsub (\Susp x \mapsto \Susp t)
\end{mathpar}

\begin{example}\label{ex:susp-id-comp}
  Applying the suspension on the terms defined in Example~\ref{ex:id-comp}, we
  obtain higher dimensional identities and compositions. More precisely, we may
  derive terms:
  \begin{align*}
    \typing{\Susp^n\Gamma_{\id}}{\Susp^{n}(\id x)}[\arr{x}{x}] &&
    \typing{\Susp^n\Gamma_{\comp{}{}}}{\Susp^{n}(\comp{f}{g})}[\arr{x}{z}]
  \end{align*}
  which we will denote by \(\id x\) and \(\comp{f}{g}\). Both of these terms
  use all variables of their respective contexts. We may visualise the contexts
  \(\Gamma\), \(\Susp \Gamma\) and \(\Susp^2 \Gamma\) as follows:
  \[
    \begin{tikzcd}
      \bullet \ar[r] & \bullet \ar[r] &\bullet
    \end{tikzcd}
    \qquad
    \begin{tikzcd}
      \bullet & \bullet
      \arrow[""{name=0, anchor=center, inner sep=0}, curve={height=-18pt}, from=1-1, to=1-2]
      \arrow[""{name=1, anchor=center, inner sep=0}, curve={height=18pt}, from=1-1, to=1-2]
      \arrow[""{name=2, anchor=center, inner sep=0}, from=1-1, to=1-2]
      \arrow[between={0.2}{0.8}, Rightarrow, from=0, to=2]
      \arrow[between={0.2}{0.8}, Rightarrow, from=2, to=1]
    \end{tikzcd}
    \qquad
    \begin{tikzcd}[column sep=huge]
      \bullet & \bullet
      \arrow[""{name=0, anchor=center, inner sep=0}, curve={height=-18pt}, from=1-1, to=1-2]
      \arrow[""{name=1, anchor=center, inner sep=0}, curve={height=18pt}, from=1-1, to=1-2]
      \arrow[""{name=2, anchor=center, inner sep=0}, curve={height=-22pt}, between={0.2}{0.8}, Rightarrow, from=0, to=1]
      \arrow[""{name=3, anchor=center, inner sep=0}, curve={height=22pt}, between={0.2}{0.8}, Rightarrow, from=0, to=1]
      \arrow[""{name=4, anchor=center, inner sep=0}, between={0.2}{0.8}, Rightarrow, from=0, to=1]
      \arrow[between={0.2}{0.8}, Rightarrow, scaling nfold=3, from=3, to=4]
      \arrow[between={0.2}{0.8}, Rightarrow, scaling nfold=3, from=4, to=2]
    \end{tikzcd}
  \]
\end{example}

\noindent
The \emph{disks} and \emph{spheres} are two families of contexts defined
recursively by:
\begin{align*}
  \disk{0} &= (\dvar_0\ty \obj) & \disk{n+1} &= \Susp\disk{n} &
  \sphere{-1} &= \emptycontext & \sphere{n+1} &= \Susp\sphere{n}
\end{align*}
We will denote the top-dimensional variable of \(\disk n\) by \(\dvar_n\). We
get a substitution \(\typing{\disk n}{\sphereInc_n}[\sphere {n-1}]\) for every
\(n\) by suspending the empty substitution
\(\typing{\disk 0}{\emptysub}[\sphere {-1}]\). The universal property of these
contexts is given by the following lemma. Note that as
opposed to Lemma~\ref{lemma:familial-representability}, we do not need to
quotient under judgemental equality here as the theory \catt does not have any
computation rules.

\begin{lemma}[{\cite[Lemma~2.3.8]{benjaminGlobularWeak$omega$categories2024}}]
  \label{lemma:sub-to-disk}
  There exists natural bijections between types \(\typing{\Gamma}{A}\)
  of dimension \(n\) and substitutions
  \(\typing{\Gamma}{\classifierty{A}}[\sphere n]\), and
  between terms \(\typing{\Gamma}{t}[A]\) of
  dimension \(n\) and substitutions
  \(\typing{\Gamma}{\classifiertm{t}}[\disk n]\), such that
  \(\classifierty{A} = \sphereInc_n\circ\classifiertm{t}\) for every term
  \(\typing{\Gamma}{t}[A]\).
\end{lemma}


\section{Weak \texorpdfstring{\(\omega\)}{ω}-categories}
\label{sec:semantics-catt}

We view \catt as a
\emph{generalised algebraic
theory}~\cite{cartmellGeneralisedAlgebraicTheories1986} and we use the framework
of \emph{clans}~\cite{joyalNotesClansTribes2017,
taylorRecursiveDomainsIndexed1987} to
describe its models. This notion of model of \catt coincides with the one
used previously~\cite[Lemma~1.1.3]{benjaminGlobularWeak$omega$categories2024}.

\begin{definition}\label{def:clan}
  A \emph{clan} is a category \(C\) with a terminal object, equipped with a
  class of arrows containing every isomorphism and every morphism to a terminal
  object, and that is stable under composition and base change. We will call the
  arrows in this class \emph{display maps} and denote them by \(\dismap\).
\end{definition}

\noindent
Stability under base change means that for every display map
\(p\colon A\dismap B\) and every morphism \(f\colon C\to B\) the pullback
\(A\times_B C\) exists and \(A\times_B C \dismap C\)
is a display map.

\begin{definition}\label{def:clan-morphism}
  A \emph{morphism of clans} \(C\to D\) is a functor preserving the terminal
  objects, the class of display maps and pullbacks along display maps. A
  \emph{transformation} between morphisms of clans is a natural
  transformation. We will denote the strict \(2\)\-category of clans
  by \(\Clan\).
\end{definition}

\begin{example}\label{ex:syntactic-clan}
  The \emph{syntactic category} \(\Syn{\mathbb T}\) of a dependent type theory
  \(\mathbb T\) consists of contexts and substitutions up to judgemental
  equality. It is a clan with display maps the closure under composition and
  isomorphisms of the weakening maps of the form
  \(\typing{\Gamma\extctx A}{\display_{\Gamma,A}}[\Gamma]\).
\end{example}

\begin{definition}\label{def:clan-models}
  A finitely complete category \(C\) is a clan with every morphism a display
  map. In particular, this gives a clan structure on \(\Set\). We then define
  the category of \emph{models} of a type theory \(\mathbb T\) to be
  \(\Mod(\mathbb{T}) = \Clan(\Syn{\mathbb T}, \Set)\).
\end{definition}

Since every morphism in \(\Set\) is a display map and pullback squares may be
pasted together, a model of a type theory \(\mathbb T\) is precisely a functor
\(F\colon \Syn{\mathbb T} \to \Set\) that preserves the terminal object and
pullback along the weakening maps \(\display_{\Gamma,A}\). There is a fully
faithful embedding:
\begin{align*}
  \interp{-} \colon \Syn{\mathbb T}^{\op} &\to \Mod(\mathbb T)  &
  \interp{\Gamma}(\Delta) &= \Syn{\mathbb T}(\Gamma, \Delta)
\end{align*}
The category of models is a reflective subcategory of functors from
\(\Syn{\mathbb T}\) to \(\Set\) closed under filtered colimits, so in
particular, it is locally finitely
presentable~\cite[Remark~2.9]{jonasDualityClans2025}.

The syntactic category of \catt is dually equivalent to the category of finite
computads~\cite{benjaminCaTTContextsAre2024}, while its category of models is
equivalent to the category of Leinster's weak
\(\omega\)\-categories~\cite{leinsterHigherOperadsHigher2003,
benjaminGlobularWeak$omega$categories2024}. These weak \(\omega\)\-categories
are a variant of Batanin's~\cite{bataninMonoidalGlobularCategories1998} and
they are models of a Grothendieck-Maltsiniotis
\emph{coherator}~\cite{grothendieckPursuingStacks1983,
maltsiniotisGrothendieck$infty$groupoidsStill2010}, as shown by
Ara~\cite{ara$infty$groupoidesGrothendieckVariante2010} and
Bourke~\cite{bourkeIteratedAlgebraicInjectivity2020}.

\begin{definition}\label{def:omega-cat}
  The category \(\omegacat\) is the
  category \(\Mod(\catt)\) of models of \catt.
\end{definition}

We define the \emph{\(n\)\-cells} of an \(\omega\)\-category \(X\) to be
morphisms \(\interp{\disk{n}}\to X\). The \emph{source} and
\emph{target} functions \(\src,\tgt\colon X_{n+1}\to X_n\) are those given by
composition with the substitutions \(\disk{n+1}\to\disk{n}\) corresponding under
\Cref{lemma:sub-to-disk} to the source and target of the top-dimensional
variable \(\dvar_{n+1}\) of \(\disk{n+1}\) respectively. We write \(x\ty \obj\)
to denote that \(x\in X_0\), and \(x\ty \arr{y}{z}\) to denote that \(x\) has
 source \(y\) and target \(z\). By \Cref{lemma:sub-to-disk}, weakening
 substitutions in \catt are pullbacks of the form:
\begin{equation}\label{eq:pb-square-sphereInc}
  \begin{tikzcd}
    {\Gamma\extctx(x \ty A)} & {\disk{n}} \\
    \Gamma & {\sphere{n-1}}
    \arrow["{\classifiertm{x}}", from=1-1, to=1-2]
    \arrow["{\display_{\Gamma,A}}"',  -{Triangle[open]}, from=1-1, to=2-1]
    \arrow["\lrcorner"{anchor=center, pos=0.125}, draw=none, from=1-1, to=2-2]
    \arrow["{\sphereInc_n}",  -{Triangle[open]}, from=1-2, to=2-2]
    \arrow["{\classifierty{A}}"', from=2-1, to=2-2]
  \end{tikzcd}
\end{equation}
so a functor \(X\colon \Syn{\catt}\to \Set\) is a model exactly when it
preserves terminal and pullback squares of the
form~\eqref{eq:pb-square-sphereInc}. It follows by structural induction that
morphisms of models \(f\colon X\to Y\) are completely determined by their values
on cells of \(X\), and that they are invertible if and only if they are
bijective on cells.

\begin{definition}\label{def:ops-id-comp}
  Every term \(\typing{\Gamma}{t}[A]\) of dimension \(n\) gives rise to an
  \emph{operation}:
  \begin{align*}
    \interp{t} & \colon \omegacat(\interp{\Gamma},X) \to X_n
    & \interp{t}(f) & = f \circ \interp{\classifiertm{t}}
  \end{align*}
  In particular, the terms of \Cref{ex:susp-id-comp} give rise to operations
  with the following boundaries:
  \begin{align*}
    \begin{aligned}
      \id &\colon X_n \to X_{n+1} \\
      \comp{-}{-} &\colon X_n\times_{X_{n-1}}X_n \to X_n
    \end{aligned} &&
    \begin{gathered}
      \src(\id x) = \tgt(\id x) = x \\
      \src(\comp{f}{g}) = \src f \qquad\qquad
      \tgt(\comp{f}{g}) = \tgt g
    \end{gathered}
  \end{align*}
\end{definition}

\emph{Invertible cells} in \(\omega\)\-categories allow us to define
equivalences within and between \(\omega\)\-categories. A morphism of
\(\omega\)\-categories \(f\colon X\to Y\) is called a \emph{weak equivalence}
if it is surjective up to invertible cells on objects and on iterated hom-sets.
It is called an \emph{equifibration} if for every \(x\in X_n\) and every
invertible cell \(g\colon f(x)\to y\), there exists an invertible
cell \(g'\) with \(f(g') = g\). These two classes of maps play an important
role in the homotopy theory of strict
\(\omega\)\-categories~\cite{lafontFolkModelStructure2010} and they are expected
to do the same for weak ones.

\begin{definition}\label{def:invertibility-structure}
  An \emph{invertibility structure} on a cell \(f\colon x\to y\) in an
  \(\omega\)\-category \(X\) is coinductively defined to consist of cells:
  \begin{align*}
    f^{\leftinv}&\ty y \to x &
    f^{\rightinv}&\ty y \to x &
    f^{\leftunit}&\ty \comp{f^{\leftinv}}{f} \to \id y &
    f^{\rightunit}&\ty \comp{f}{f^{\rightinv}} \to \id x
  \end{align*}
  and invertibility structures on \(f^{\leftunit}\) and \(f^{\rightunit}\). We
  say that \(f\) is \emph{invertible} if it admits an invertibility structure.
\end{definition}

We note that this notion of invertible cells is a priori weaker than the usual
one where the left inverse \(f^{\leftinv}\) is strictly equal to the right
inverse \(f^{\rightinv}\). Nonetheless, the two notions of invertibility
coincide~\cite{riceCoinductiveInvertibilityHigher2020,
fujii$o$equifibrationsStrictWeak2025}. Our use of this definition is inspired
by work on the \emph{walking equivalence}~\cite{
hadzihasanovicModelCoherentWalking2024} and equifibrations~\cite{
fujii$omega$weakEquivalencesWeak2024}. In analogy with homotopy type
theory~\cite[Theorem~4.3.2]{
theunivalentfoundationsprograminstituteofadvancedstudyHomotopyTypeTheory2013},
we expect that requiring separate left and right inverses ought to make
invertibility structures on a cell \(f\) unique.

An important property of invertible cells is that they are closed under the
operations of weak \(\omega\)\-categories~\cite{fujiiWeaklyInvertibleCells2023},
in the sense of the following theorem. More precisely, composites of invertible
cells, as well as all identity cells and higher coherences like associators and
unitors, are again invertible~\cite{
  benjaminInvertibleCells$omega$categories2024}.

\begin{restatable}{theorem}{thminvertibility}
  \label{thm:invertibility-catt}
  \cite[Proposition~16~\&~Theorem~33]{benjaminInvertibleCells$omega$categories2024}
  Let \(\fulltyping{\Gamma}{A}\) be a full type of dimension \(n\) on a pasting
  context, and let \(\gamma\colon \interp{\Gamma} \to X\) be a morphism. Then
  the cell \(\interp{\coh_{\Gamma,A}}(\gamma)\) is invertible in \(X\) when
  \(\interp{x}(\gamma)\) is invertible for every \((n+1)\)\-dimensional variable
  \(x\) of \(\Gamma\).
\end{restatable}

The proof of the theorem coinductively constructs an invertibility structure on
cells of that form. When \(\dim \Gamma \le n\), the condition on \(\Gamma\) is
vacuous and the inverses and cancellators are again coherences over the same
pasting context \(\Gamma\). The case \(\dim \Gamma = n + 1\) corresponds to
composites of invertible cells. When \(X = \interp{\Delta}\) is representable,
the converse also holds:

\begin{proposition}\label{prop:invertible-cells-catt-context}
  \cite[Proposition~37]{benjaminInvertibleCells$omega$categories2024}
  A term \(\typing{\Delta}{t}[A]\) of dimension \(n\) is invertible in
  \(\interp{\Delta}\) if and only if it is of the form
  \(t = \coh_{\Gamma,A}[\gamma]\) with \(x[\gamma]\) invertible for every
  \(n\)\-dimensional variable \(x\) of \(\Gamma\) if and only if it does not
  use any \(n\)\-dimensional variables of \(\Delta\).
\end{proposition}




\section{The type theory \texorpdfstring{\cattinv}{ICaTT}}
\label{sec:cattinv}

We define the extension \cattinv of \catt to allow equipping terms with
invertibility structures, and show that it is normalising on a large class of
terms. We extend first the raw syntax of \catt with a single
type constructor \(\Inv_{A}(u)\) where \(A\) is a type and \(u\) is a term, and
with nine term constructors listed below, where \(\gamma\) is a substitution,
\(\{e_{x}\}\) is a family of terms indexed by a finite set,
 and the remaining arguments are terms:
\begin{mathpar}
  \can(t,\{e_{x}\}) \and
    \coind(t,t_{l},t_{r},t_{lu},t_{ru},t_{ilu},t_{iru}) \and
    \rec(t,t_{l},t_{r},t_{lu},t_{ru},t_{ilu},t_{iru},\gamma) \\
    \leftinv(t)
    \and \rightinv(t)
    \and \leftunit(t)
    \and \rightunit(t)
    \and \leftunitwitness(t)
    \and \rightunitwitness(t)
\end{mathpar}
The suspension operation and the action of raw substitutions are extended to
these new types and terms by letting them compute under the new constructors,
as described in Appendix~\ref{app:cattinv}. The type theory \cattinv inherits every rule
of \catt. The introduction rules for the type \(\Inv\) and its main constructors
and dual destructors are given below:
\begin{mathpar}
  \inferdef[\(\Inv\)-in]{
    \typing{\Gamma}{t}[\arr[A]{u}{v}]
  }{
    \typing{\Gamma}{\Inv_{\arr[A]{u}{v}}(t)}
  }\label{rule:inv-intro}
  \and
  \inferdef[\(\coind\)-intro]{
    \typing{\Gamma}{t}[\arr{x}{y}]\\
    \typing{\Gamma}{t_l}[\arr{y}{x}]\\
    \typing{\Gamma}{t_r}[\arr{y}{x}]\\\\
    \typing{\Gamma}{t_{lu}}[\arr{\comp{t_l}{t}}{\id y}] \\
    \typing{\Gamma}{t_{ru}}[\arr{\comp{t}{t_r}}{\id x}] \\\\
    \typing{\Gamma}{t_{ilu}}[\Inv(t_{lu})] \\
    \typing{\Gamma}{t_{iru}}[\Inv(t_{ru})]
  }{
    \typing{\Gamma}{\coind(t,t_{l},t_{r},t_{lu},t_{ru},t_{ilu},t_{iru})}[\Inv(t)]
  }
  \label{rule:coind-in}
  \\
  \inferdef[\(\Inv\)-out]{
    \typing{\Gamma}{e}[\Inv_{\arr{u}{v}}(t)]
  }{
    \typing{\Gamma}{\leftinv(e)}[\arr{v}{u}] \\
    \typing{\Gamma}{\rightinv(e)}[\arr{v}{u}] \\\\
    \typing{\Gamma}{\leftunit(e)}[\arr{\comp{\leftinv(e)\!}{\!t}}{\id v}] \\
    \typing{\Gamma}{\rightunit(e)}[\arr{\comp{t\!}{\!\rightinv(e)}}{\id u}] \\\\
    \typing{\Gamma}{\leftunitwitness(e)}[\Inv(\leftunit(e))] \\
    \typing{\Gamma}{\rightunitwitness(e)}[\Inv(\rightunit(e))]
  }
  \label{rule:inv-dest}
\end{mathpar}
We require also the usual \(\beta\)\-reduction and \(\eta\)\-expansion rules,
presented fully in Appendix~\ref{app:cattinv}.
The introduction rule for the constructor \(\can\) is a way to internalise
\Cref{thm:invertibility-catt} in \cattinv :
\begin{mathpar}
  \inferdef[\(\can\)-in\!]{
    \typing{\Delta}{\coh_{\Gamma,A}[\gamma]\!}[\!A[\gamma]]\;\;\;
    \set{\typing{\Delta}{e_x\!}[\!\Inv(x[\gamma])]}_{x\in \Var_{\dim A +
        1}(\Gamma)} }{
    \typing{\Delta}{\can(\coh_{\Gamma,A}[\gamma],
      \set{e_x})}[\Inv(\coh_{\Gamma,A}[\gamma])]
  }
  \label{rule:can-intro}
\end{mathpar}
where \(\Var_n(\Gamma)\) denotes the set of variables of \(\Gamma\) of dimension
\(n\). We require furthermore computation rules that coinductively identify the
invertibility structure obtained by the rule with the one obtained by the
theorem. A more detailed description of these rules can be found in
Appendix~\ref{app:cattinv}.

Before stating~\ruleref{rule:rec-intro}, the introduction rule for $\rec$, we
introduce the walking equivalence. The universal property of the walking
equivalence is given by \Cref{lemma:familial-representability}, whose proof is
analogous to that of \Cref{lemma:sub-to-disk} and only uses that every rule of
\cattinv is stable under substitution.

\begin{definition}\label{def:walking-equivalence}
  The \emph{walking equivalence} for \(n\in\N\) is the context:
  \begin{align*}
    \Equiv{n+1} &= \disk{n+1}\extctx(\evar_{n+1}\ty\Inv(\dvar_{n+1}))
  \end{align*}
  We will denote its weakening substitution by
  \(\typing{\Equiv{n+1}}{\equivDisk_{n+1}}[\disk{n+1}]\).
\end{definition}

\begin{lemma}\label{lemma:familial-representability}
  There exists a natural bijection between types \({\typing{\Gamma}{A}}\) of
  \cattinv and the set of substitutions of the form
  \(\typing{\Gamma}{\classifierty{A}}[\sphere n]\) or of the form
  \(\typing{\Gamma}{\classifierty{A}}[\disk {n+1}]\) up to judgemental equality.
  Similarly, there exists a natural bijection between terms
  \(\typing{\Gamma}{t}[A]\) and substitutions of the form
  \(\typing{\Gamma}{\classifiertm{t}}[\disk n]\) or of the form
  \(\typing{\Gamma}{\classifierequiv{t}}[\Equiv{n+1}]\). Moreover, depending on
  the form they correspond to, they satisfy either
  \(\classifierty{A} = \sphereInc_n\circ\classifiertm{t}\) or
  \(\classifierty{A} = \equivDisk_{n+1}\circ\classifiertm{t}\).
\end{lemma}


The type theory \catt does not have function types, since its category of models
is not cartesian closed, which makes recursive or corecursive definitions
non-trivial. The same obstruction appears in directed homotopy type theory as
well and several approaches have been developed to circumvent it~\cite{
  riehl2017type,north2019towards,altenkirch2024synthetic,laretto2026di}. For
that reason, we model corecursive calls in our theory by extending the context
with the coinductive hypothesis. More precisely, given a term
\(\typing{\Equiv{n+1}}{t}[\arr{u}{v}]\) we define the extended context
\[
  \EquivInd{n+1}{t} = \Equiv{n+1} \extctx
  (h^-\ty \Inv(\Susp t)[\classifierequiv{\leftunitwitness(\evar_{n+1})}])\extctx
  (h^+\ty \Inv(\Susp t)[\classifierequiv{\rightunitwitness(\evar_{n+1})}])
\]
containing two new variables as the coinductive hypotheses. Given a term
\(\typing{\Equiv{n+1}}{r}[\Inv(t)]\), we define a substitution
\(\typing{\Equiv{n+1}}{\instantiation{r}}[\EquivInd{n+1}{t}]\) which
instantiates the variables of \(\EquivInd{n+1}{t}\) with corecursive calls to
the construction as follows:
\[
  \instantiation{r} = \id_{\Equiv{n+1}}\extsub
  (h^-\mapsto\Susp(r)[\classifierequiv{\leftunitwitness(\evar_{n+1})}]) \extsub
  (h^+\mapsto\Susp(r)[\classifierequiv{\rightunitwitness(\evar_{n+1})}])
\]
This allows to state the introduction and computation rules for the corecursor
\(\rec\), which are similar to the one for \(\coind\) except that the context in
the last two premises has been extended by the inductive hypotheses:
\begin{mathpar}
  \inferdef[\(\rec\)-in]{
    \typing{\Equiv{n+1}}{t}[\arr{x}{y}]\\
    \typing{\Gamma}{\gamma}[\Equiv{n+1}] \\
    \typing{\Equiv{n+1}}{t_l}[\arr{y}{x}]\\
    \typing{\Equiv{n+1}}{t_r}[\arr{y}{x}]\\\\
    \typing{\Equiv{n+1}}{t_{lu}}[\arr{\comp{t_l}{t}}{\id y}] \\
    \typing{\Equiv{n+1}}{t_{ru}}[\arr{\comp{t}{t_r}}{\id x}] \\\\
    \typing{\EquivInd{n+1}{t}}{t_{ilu}}[\Inv(t_{lu})] \\
    \typing{\EquivInd{n+1}{t}}{t_{iru}}[\Inv(t_{ru})] \\
    }{
    \typing{\Gamma}{\rec(t,t_l,t_r,t_{lu},t_{ru},t_{ilu},t_{iru},\gamma)}[\Inv(t[\gamma])]
  }\label{rule:rec-intro}
\end{mathpar}
The \(\beta\)\-rules for this corecursor are similar to the \(\beta\)\-rules for
\(\coind\). The differing ones are the last two, which are given for
\(e = \rec(t, t_{l}, t_{r}, t_{lu}, t_{ru}, t_{ilu}, t_{iru}, \gamma)\) by
\begin{align*}
  \leftunitwitness e
  &=  t_{ilu}[\instantiation{e} \circ\ \gamma]
  & \rightunitwitness e
  &=  t_{iru}[\instantiation{e} \circ\ \gamma]
\end{align*}
A complete list of these rules can be found in Appendix~\ref{app:cattinv}. While
not a complete coinductive scheme, \(\rec\) captures a limited form of
coinductive constructions controlled by the suspension operation. The resulting
coinductive definition are always productive. This can be seen for instance by
the fact that an invertibility structure only comprises invertibility structures
of higher dimensions, which may be corecursively produced by the suspending the
one being constructed. Hence, applying a destructor on such an invertibility
structure will terminate in finitely many steps of computation (see also
Proposition~\ref{prop:normalising}).

This completes the presentation of the type theory \cattinv. Stability of all
rules under substitution and suspension easily implies the following lemma:

\begin{lemma}\label{lemma:subs-susp-stability}
  The following rules are admissible in \cattinv{}:
  \begin{mathpar}
    \inferrule{
      \typing{\Delta}{\gamma}[\Gamma] \\
      \typing{\Gamma}{A}
    }{
      \typing{\Delta}{A[\gamma]}
    } \and
    \inferrule{
      \typing{\Delta}{\gamma}[\Gamma] \\
      \typing{\Gamma}{t}[A]
    }{
      \typing{\Delta}{t[\gamma]}[{A[\gamma]}]
    } \and
    \inferrule{
      \typing{\Delta}{\gamma}[\Gamma] \\
      \typing{\Gamma}{\gamma'}[\Gamma']
    }{
      \typing{\Delta}{\gamma'\circ\gamma}[\Gamma']
    } \\
    \inferrule{
      \ctxty{\Gamma}
    }{
      \typing{\Susp\Gamma}
    } \and
    \inferrule{
      \typing{\Gamma}{A}
    }{
      \typing{\Susp\Gamma}{\Susp A}
    } \and
    \inferrule{
      \typing{\Gamma}{t}[A]
    }{
      \typing{\Susp\Gamma}{\Susp t}[\Susp A]
    } \and
    \inferrule{
      \typing{\Gamma}{\gamma}[\Delta]
    }{
      \typing{\Susp\Gamma}{\Susp
        \gamma}[\Susp\Delta]
    }
  \end{mathpar}
\end{lemma}

We close this section with a few remarks on the rewriting system of \cattinv.
Contrary to \catt, the theory \cattinv has non-trivial computation rules, which
we orient in the direction of \(\beta\)\-reduction and \(\eta\)\-expansion. This
gives rise to a locally confluent but non-terminating rewriting system. Local
confluence may be checked by considering all critical pairs -- for instance,
considering an invertibility structure \(e = \can(t,\{e_{x}\})\), we have the
following conversion starting from the term \(\leftinv(e)\):
\[
  \begin{tikzcd}
    & \leftinv(e) \ar[dl, "\beta"'] \ar[dr,"\eta"] & \\
    t^{\leftinv}\ar[dr,equal, dashed]
    & &
        \leftinv(\coind(t,\leftinv(e),\ldots) \ar[d, "\beta", dashed]\\
    & t^{\leftinv} & \leftinv(e))\ar[l,"\beta", dashed]
  \end{tikzcd}
\]
The other critical pairs can be handled similarly. Non-termination is not
surprising, since the constructors \(\rec\) and \(\can\) are designed to
syntactically capture infinite applications of the \(\coind\) constructor. We
argue that this is not an issue for our theory, since no rule of \cattinv
requires checking equality of terms of type \(\Inv\), and the theory is
normalising on terms of type \(\obj\) or \(\arr{u}{v}\), which we will call
\emph{categorical terms}.

\begin{proposition}\label{prop:normalising}
  The rewrite system of \cattinv is
  normalising on categorical terms, in that there is a
  procedure \(\nf\) choosing a representative out of each
  \(\beta\eta\)\-equivalence class.
\end{proposition}

\begin{proof}
  We prove this by defining a normalisation strategy. To this end, we introduce
  an intermediate rewrite system, where we guard \(\eta\)-expansion by the
  dimension. We call \(n\)-dimensional \(\eta\)-expansion the rewrite system
  where we only \(\eta\)-expand invertibility structures of dimension at most
  \(n\). We also restrict \(\eta\)-expansion to not apply under any destructor,
  as done by Jay and Ghani~\cite{jayVirtues1995}. The restricted system is
  terminating, and remains locally confluent.
  Two \(n\)\-dimensional categorical terms are \(\beta\eta\)\-equivalent exactly
  when they have the same \(n\)-dimensional normal form. Indeed, no
  invertibility structure of dimension above \(n\) may appear in a
  \(\beta\)-reduced form, thus \(\eta\)-expansion of invertibility structures of
  dimension more than \(n\) is irrelevant. Hence, \(\nf(t)\) is the normal form
  for the \((\dim t)\)-dimensional restricted system for each categorical term
  \(\typing{\Gamma}{t}[A]\).
\end{proof}

For the rest of this article, it is useful to notice that categorical types and
terms in normal form may not use the constructors \(\can\), \(\rec\) or
\(\coind\), as those are eliminated by the normalisation procedure. Thus, the
invertibility structures that may appear in such a term are variables, or
iterated applications of the destructors \(\leftunitwitness\) and
\(\rightunitwitness\) on those.


\section{Internal proofs using invertibility}
\label{sec:examples}

We have developed a prototype implementation of the type theory \cattinv in the
form of a proof assistant allowing to reason about invertibility structures in
\(\omega\)-categories. In this section, we discuss the implementation and
present several examples of recent results of the literature that we formalised
in it.

\subsection{Implementation and additional features}

We have written our implementation on top of an existing implementation of
\catt, allowing us to leverage and extend several features already present in
said implementation. In particular, we have implemented a unification algorithm
for \cattinv, allowing to resolve implicit arguments. The unification follows a
best-effort principle and is incomplete, but it is enough to significantly
reduce the size of the terms one needs to write in practice. Additionally one
may always write full terms manually if unification fails. In practice, the
implementation infers automatically which arguments can be made explicit, and
one may also use the wildcard \verb|_| for an argument which, despite explicit
in general, can be inferred in a particular instance. The existing
implementation of \catt has two keywords for declaring names
\begin{verbatim}
coh [name] [ctx] : [ty]
let [name] [ctx] : [ty] = [tm]
\end{verbatim}
The keyword \verb|coh| declares ``coherences'', which are a pair of a pasting
diagram and a full type in it. In our concrete syntax, these are not terms, as
they require a substitution to become one, rather they are schemas of
definitions corresponding to a particular form of cut. The keyword \verb|let|
declares arbitrary terms, obtained as variables, or coherence or terms applied
to a substitution. In our implementation of \cattinv, we add the type
constructor \verb|Inv|, together with two new keywords
\begin{verbatim}
inv [name] [ctx]={[tm],[tm],[tm],[tm],[tm],[tm],[tm]}
corec [name] [ctx]={[tm],[tm],[tm],[tm],[tm],[tm],[tm]}
\end{verbatim}
The keyword \verb|inv| declares terms built with the \(\coind\) constructor,
while the keyword \verb|corec| declares corecursive definitions: It is required
in the latter that the context is of the form \(\Equiv{n+1}\), and we provide
the keywords \verb|IHleft| and \verb|IHright| to access the coinductive
hypotheses within the last two arguments of corecursive definitions. The keyword
\verb|corec| is similar to \verb|coh| as it does not correspond directly to a
term constructed with the constructor of the same name, but rather to an
application of cut which needs a substitution to produce a term. Finally, we
also provide term constructor \verb|can|, which takes as argument a term and a
list of invertibility structures, written as comma-separated terms delineated by
curly braces.

The suspension operation of \catt and \cattinv is left implicit, which means
that every definition we write is thought of as a family of constructions
generated by suspension. We also use the built-in keywords \verb|id| which
denotes the identity coherence, and \verb|comp| which computes the compositions
along codimension \(1\) faces. The keyword \verb|comp| is particular in the
sense that it is multi-ary: it can be used with any number of arguments and will
generate the unbiased composite of said arguments.

\subsection{Examples of formalised results}

We present a few results that we have formally proven in \cattinv. We give
statements in proposition environments and only provide sketches for some of the
proofs. The full proofs can be found in the in the file
\verb|proofs/invertibility.catt| of the supplementary
material\footnote{available at \url{https://doi.org/10.5281/zenodo.23686017}}.

\begin{restatable}{proposition}{compinv}
  The composite of two invertible cells is invertible.
\end{restatable}
\begin{proof}
  This is a straightforward application of \ruleref{rule:can-intro}. Note that
  we state the result in dimension \(1\), but stability under
  suspension allows us to derive it for any dimension.
\begin{verbatim}
let compinv (x : *) (y : *) (z : *) (f : x -> y) (g : y -> z)
            (e : Inv (f)) (e' : Inv (g))
: Inv (comp f g) = can ( comp f g { e , e' })
\end{verbatim}
\vspace{-1.5\baselineskip}
\end{proof}

\begin{restatable}{proposition}{invisinv}  \label{prop:left-inv}
  The left and right inverses of any invertibility structure are invertible.
\end{restatable}

\begin{proof}
We first define a cell relating the left and right inverse, illustrated in
\Cref{fig:lri}.
\begin{verbatim}
let lri (x : *) (y : *) (f : x -> y) (e : Inv(f))
: linv (e) -> rinv (e)
\end{verbatim}
We then construct a witness that the above cell is also invertible, provided
that the left inverse of \(\leftunitwitness (e)\) is invertible:
\begin{verbatim}
let lriU-aux (x : *) (y : *) (f : x -> y)
             (e : Inv(f)) (e' : Inv (linv (irunit (e))))
: Inv (lri e)
\end{verbatim}
The invertibility structure on the left inverse is then obtained as a corecursive
definition:
\begin{verbatim}
corec linv-inv (x : *) (y : *) (f : x -> y) (e : Inv(f))
= { linv(e) , f , f , comp (whiskl f (lri e)) (runit (e)) ,
    lunit (e) ,
    can (_ {can (_ { lriU-aux IHright }) ,
    (irunit (e))}), ilunit (e) }
\end{verbatim}
The proof for the right inverse is symmetric.
\end{proof}

\begin{figure}[ht]
  \centering
  \begin{equation*}
    \begin{split}
      \leftinv(e)
        \longrightarrow \leftinv(e) \ast \id
        &\longrightarrow \leftinv(e) \ast (f \ast \rightinv(e)) \\
        &\longrightarrow (\leftinv(e) \ast f) \ast \rightinv(e)
        \longrightarrow \id\ast\rightinv(e)
        \longrightarrow \rightinv(e)
    \end{split}
  \end{equation*}
  \caption{Relating left inverse to right inverse}
  \label{fig:lri}
\end{figure}

\begin{restatable}{proposition}{leftisright}
  There is an invertible cell between the left and right
  inverse of a term \(t\).
\end{restatable}

\begin{restatable}{proposition}{transport}
  \label{prop:transport}
  Given an invertible cell \(f \ty u \to v\), if \(u\) is invertible, then so is
  \(v\).
\end{restatable}

\begin{restatable}{proposition}{twoofsix}
  \label{prop:2of6}
  Invertible cells satisfy the 2-of-6 property.
\end{restatable}


\section{Applications to \texorpdfstring{\(\omega\)}{ω}-categories}
\label{sec:cattinv-semantics}

Every context and substitution of \catt is also one of \cattinv, so there exists
a morphism of clans \(\syntaxincl\colon\Syn{\catt}\to\Syn{\cattinv}\), which
gives rise to an adjunction~\cite[Section~4]{jonasDualityClans2025} on the
level of models by precomposition and left Kan extension:
\[
  \begin{tikzcd}
    {\syntaxincl^*\colon \Mod(\cattinv)} & {\omegacat \colon \syntaxincl_!}
    \arrow[""{name=0, anchor=center, inner sep=0},
      shift left=1.2, from=1-1, to=1-2]
    \arrow[""{name=1, anchor=center, inner sep=0},
      shift left=1.2, from=1-2, to=1-1]
    \arrow["\dashv"{anchor=center, rotate=90}, draw=none, from=1, to=0]
  \end{tikzcd}
\]
Our first goal is to show that \(\syntaxincl\) is fully faithful and that
\(\cattinv\) is conservative over \(\catt\).

\begin{proposition}\label{prop:nf-in-cattctx}
  The normal forms of any categorical type or term of \cattinv over a context
  of \catt are valid types or terms of \catt.
\end{proposition}
\begin{proof}
  We recall that the only invertibility structures that may appear in \(\nf(A)\)
  and \(\nf(t)\) are either variables or iterated applications of destructors
  \(\leftunitwitness\) and \(\rightunitwitness\) on those. Since \(\Gamma\) is a
  context of \catt, it contains no variables that are invertibility structures.
  Thus neither \(\nf(A)\) nor \(\nf(t)\) may contain invertibility structures,
  and therefore, they are valid types and terms in \catt.
\end{proof}

\begin{corollary}\label{thm:conservativity}
  The theory \cattinv is a conservative extension of \catt.
\end{corollary}
\begin{proof}
  For every type \(\typing{\Gamma}{A}\) in \catt and term
  \(\typing{\Gamma}{u}[A]\) in \cattinv, \Cref{prop:nf-in-cattctx} implies
  that \(\typing{\Gamma}{\nf(u)}[\nf(A)]\). As \(A\) is a \catt type, it
  is already in normal form, so it was already
  inhabited in \catt.
\end{proof}

\begin{proposition}\label{prop:incl-ff}
  The functor \(\syntaxincl\ty\Syn{\catt}\to\Syn{\cattinv}\) is fully
  faithful.
\end{proposition}
\begin{proof}
  We fix a context \(\Gamma\) of \catt and show by induction on \(\Delta\) that
  \(\syntaxincl\) is bijective between substitutions \(\Gamma\to \Delta\) in the
  two theories. The base case is given by the fact that \(\syntaxincl\)
  being a morphism of clans, so suppose that \(\Delta = \Delta'\extctx(x\ty A)\)
  is an extended context. Given a substitution
  \(\typing{\Gamma}{\gamma\extsub(x\mapsto u)}[\Delta]\) in \cattinv, there
  exists unique \(\typing{\Gamma}{\bar{\gamma}}[\Delta']\) in \catt such that
  \(\syntaxincl\bar{\gamma} = \gamma\) by induction. Then the unique preimage of
  \(\gamma\extsub(x\mapsto u)\) under \(\syntaxincl\) is
  \(\typing{\Gamma}{\bar{\gamma}\extsub(x\mapsto \nf(u))}[\Delta]\), which is a
  valid substitution in \catt by \Cref{prop:nf-in-cattctx}
\end{proof}

\begin{corollary}\label{coro:minterp-restriction}
  The following diagram commutes, where \(\interp{-}\) and \(\minterp{-}\)
  denote the embeddings of \catt and \cattinv into their respective categories
  of models:
  \[
    \begin{tikzcd}
      \Syn{\catt}^{\op} \ar[r, "\syntaxincl^{\op}"] \ar[d, "\interp{-}"']
      & \Syn{\cattinv}^{\op} \ar[d, "\minterp{-}"] \\
      \omegacat \ar[from=r, "\syntaxincl^*"'] & \Mod(\cattinv)
    \end{tikzcd}
  \]
\end{corollary}

Ozornova and Rovelli~\cite{ozornovaWhatEquivalenceHigher2024} gave an inductive
description of a weak \(\omega\)\-category \(\Equiv{1}_{\OR}\) that they called
the walking equivalence, which was subsequently used to characterise
equifibrations via a lifting property~\cite{
fujii$o$equifibrationsStrictWeak2025}. Via the equivalence between \catt
contexts and finite computads~\cite{benjaminCaTTContextsAre2024}, we can give a
description of \(\Equiv{1}_{\OR}\) as the colimit of a sequence of contexts of
\catt. More precisely, we define a sequence of contexts \(\Equiv{1,n}\) of
\catt recursively with substitutions:
\begin{align*}
  \typing{\Equiv{1,n}}{i^{n}}[\Equiv{1,n-1}] &&
  \typing{\Equiv{1,n}}{f^{n}}[\Susp\Equiv{1,n-1}] &&
  \typing{\Equiv{1,n}}{g^{n}}[\Susp\Equiv{1,n-1}]
\end{align*}
with \(\iota^n\) a display map, starting from the base cases
\(\Equiv{1,0} = \sphere{0} = (x\ty \obj)\extctx (y\ty \obj)\) and
\begin{align*}
  \Equiv{1,1} &= \Equiv{1,0} \extctx(u \ty \arr{x}{y}) \extctx (v\ty \arr{y}{x})
                \extctx (w\ty \arr{y}{x}) &
  i^{1} &= (x\mapsto x) \extsub(y \mapsto y) \\
  f^{1} &= (v^- \mapsto x)\extsub(v^+ \mapsto x)\extsub(x \mapsto
          v\ast u) \extsub(y \mapsto \id y) \\
  g^{1} &= (v^- \mapsto y)\extsub(v^+ \mapsto y)\extsub(x \mapsto
          u\ast w) \extsub(y \mapsto \id x)
\end{align*}
where \(v^\pm\) are the new variables introduced by the suspension. For the
inductive case, we form the following limit in \(\Syn{\catt}\), which exists
since suspension preserves display maps:
\begin{equation}\label{eq:def-e1n}
\begin{tikzcd}
  \Susp \Equiv{1,n}\ar[d,"\Susp{i^{n}}"', -{Triangle[open]},]
  & \Equiv{1,n+1}
    \ar[d,dashed,"i^{n+1}", -{Triangle[open]},]
    \ar[l, dashed, "f^{n+1}"']
    \ar[r, dashed, "g^{n+1}"]
  & \Susp\Equiv{1,n} \ar[d,"\Susp{i^{n}}", -{Triangle[open]},]\\
  \Susp{\Equiv{1,n-1}}
  & \Equiv{1,n}
    \ar[l, "f^{n}"]
    \ar[r, "g^{n}"']
  & \Susp{\Equiv{1,n-1}}
\end{tikzcd}
\end{equation}
We note that the size of these contexts increases exponentially. The walking
equivalence \(\Equiv{1}_{\OR}\) is then the colimit of the
\(\omega\)\-categories presented by these contexts:
\begin{equation}\label{eq:colimit-equiv-OR}
  \Equiv{1}_{\OR} = \colim\left(\
    \begin{tikzcd}[cramped]
      \interp{\Equiv{1,0}} \ar[r,"\interp{i^{1}}"]
      & \interp{\Equiv{1,1}} \ar[r,"\interp{i^{2}}"]
      & \interp{\Equiv{1,2}} \ar[r,"\interp{i^{3}}"]
      & \ldots
    \end{tikzcd}\
    \right)
\end{equation}
This colimit is a computad, but it cannot be described as a context of \catt,
since it has infinitely many generators. Nonetheless, it can be presented by a
context of \cattinv{}:

\begin{theorem}\label{thm:colimit}
  There is an isomorphism of \(\omega\)\-categories:
  \(
    \syntaxincl^{*}\minterp{\Equiv{1}} \cong \Equiv{1}_{\OR}
  \)
\end{theorem}

\begin{proof}
  We construct a cone \(\typing{\Equiv{1}}{\gamma^{n}}[\Equiv{1,n}]\) in
  the syntactic category of \cattinv such that
  \begin{align*}
    i^{n}\circ\gamma^{n} &= \gamma^{n-1} &
    f^{n} \circ \gamma^{n} &= \Susp\gamma^{n-1} \circ
                             \classifierequiv{\leftunitwitness(\evar_{1})} &
    g^{n} \circ \gamma^{n} &= \Susp\gamma^{n-1} \circ
                             \classifierequiv{\rightunitwitness(\evar_{1})}
  \end{align*}
  by defining in the base cases:
  \begin{align*}
    \gamma^{0}
      &= (x \mapsto \dvar_{0}^{-}) \extsub (y \mapsto \dvar_{0}^{+}) &
    \gamma^{1}
      &= \gamma^{0} \extsub (u \mapsto \dvar_{1}) \extsub (v \mapsto
          \leftinv(\evar_{1}))\extsub(w\mapsto \rightinv(\evar_{1}))
  \end{align*}
  and using the universal property of the limit \eqref{eq:def-e1n} in the
  inductive case.

  By universality of the colimit \eqref{eq:colimit-equiv-OR}, we obtain a
  morphism \(\gamma\colon\Equiv{1}_{\OR}\to\syntaxincl^{*}\minterp{\Equiv{1}}\),
  which we need to show is invertible. That is, we need to show that for every
  context \(\Gamma\) of \(\catt\), the following is a bijection:
  \[
    \gamma_{\Gamma} \ty \Equiv{1}_{\OR}(\Gamma) \to \Syn{\cattinv}(\Equiv{1}, \Gamma)
  \]
  Since filtered colimits of models are computed pointwise~\cite{jonasDualityClans2025},
  we have that:
  \[
    \Equiv{1}_{\OR}(\Gamma) \cong \colim\Syn{\cattinv}(\Equiv{1,n},\Gamma).
  \]
  Therefore, it suffices to show that for \(n \geq \dim \Gamma\), the following
  morphism is a bijection:
  \[
    \minterp{\gamma^{n}}_{\Gamma} \ty \Syn{\cattinv}(\Equiv{1,n}, \Gamma) \to \Syn{\cattinv}(\Equiv{1},\Gamma)
  \]
  We prove this by strong induction on \((n, \dim \Gamma)\) lexicographically
  ordered. For fixed \(n\), the case where \(\dim \Gamma < n\) follows by
  \[
    \minterp{i^{n}}_\Gamma \ty \Syn{\cattinv}(\Equiv{1,n-1},\Gamma) \to \Syn{\cattinv}(\Equiv{1,n},\Gamma)
  \]
  being bijective and by the induction hypothesis for \((n-1,\dim \Gamma)\), so
  we suppose that \(\dim \Gamma = n\) and proceed by structural induction.

  Suppose first that \(\Gamma = \disk{n}\) and let
  \(\typing{\Equiv{1}}{\classifiertm{t}}[\Gamma]\) a substitution.

  If the normal form of \(t\) is a neutral, then there exists a unique variable
  of \(\Equiv{1,n}\) mapping onto it by Claim~\ref{claim}. If the normal form of
  \(t\) is \(\coh_{\Delta,A}[\delta]\), then by the inductive hypotheses, there
  exists unique map \(\delta'\) such that \(\delta = \delta' \circ\ \gamma^n\).
  Hence \(\coh_{\Delta,A}[\delta']\) is the unique preimage of \(t\) under
  \(\gamma^n\), since no variable is sent to a coherence by \(\gamma^n\).

  Suppose then that \(\Gamma = \Gamma'\extsub (x\ty A)\) is an arbitrary
  context of dimension \(n\ge 0\), and let
  \(\typing{\Equiv{1}}{\sigma\extsub(x\mapsto t)}[\Gamma]\) be a substitution.
  By structural induction, we may assume that there exist unique \(\sigma'\)
  and \(t'\) such that
  \(\sigma = \sigma' \circ \gamma^n\) and \(t = t'[\gamma^n]\).
  By uniqueness of morphisms out of \(\sphere{k}\) for \(k < n\), we have that
  \(\typing{\Equiv{1,n}}{t'}[{A[\sigma']}]\). Therefore,
  \(\Equiv{1,n}{\sigma'\extsub(x\mapsto t')}[\Gamma]\) is the unique preimage of
  \(\sigma\extsub(x\mapsto t)\) under \(\minterp{\gamma^{n}}(\Gamma)\).
\end{proof}

\begin{claim}\label{claim}
  In the notation above, substitutions
  \(\typing{\Equiv{1}}{\gamma^n}[\Equiv{1,n}]\) induce a bijection between
  variables of \(\Equiv{1,n}\) and neutral categorical terms of \(\Equiv 1\) of
  dimension at most~\(n\).
\end{claim}
\begin{proof}
  We proceed by induction on \(n\). The only neutrals of dimension \(n=0\) are
  the variables \(\dvar_0^\pm\) from which the base case follows. The
  categorical neutrals of dimension \(n=1\) are precisely the variable
  \(\dvar_1\) and its inverses \(\leftinv(\evar_1)\) and \(\rightinv(\evar_1)\)
  proving the claim for \(n=1\). Suppose then that \(n\ge 2\).

  By induction, we can show that there exist exactly \(2^{n-1}\) neutrals of
  type \(\Inv\) in dimension \(n\), since they are of the form
  \(\leftunitwitness(e)\) and \(\rightunitwitness(e)\) for \(e\) a neutral
  invertibility structure one dimension lower. Hence, there are
  \(3\cdot 2^{n-1}\) categorical neutrals of dimension \(n\) obtained either by
  \(\leftinv(e)\) and \(\rightinv(e)\) for \(e\) a neutral of dimension \(n\),
  or by \(\leftunit(e')\) and \(\rightunit(e')\) for \(e'\) neutral of dimension
  \((n-1)\).

  The top-dimensional variables of \(\Equiv{1,n}\) are the disjoint union of two
  copies of these of \(\Susp\Equiv{1,n-1}\). By induction, it follows that there
  exist exactly \(3\cdot 2^{n-1}\) of them. These are sent injectively to
  disjoint sets of neutrals by the substitutions
  \(\classifierequiv{\leftunitwitness(\evar_1)}\circ \Susp \gamma^{n-1}\) and
  \(\classifierequiv{\rightunitwitness(\evar_1)}\circ \Susp \gamma^{n-1}\) due
  to the inductive hypothesis and preservation of neutrals by the suspension. It
  follows that \(\gamma^{n}\) maps variables to neutrals injectively, and hence
  bijectively.
\end{proof}

\begin{corollary}\label{coro:equifibration-classifier}
  The equifibrations are the morphisms with the right lifting property
  against \(\syntaxincl^{*}\minterp{\classifiertm{\src\dvar_{n}}} \ty
  \interp{\disk{n}}\to\syntaxincl^{*}\minterp{\Equiv{n+1}}\)
  for all \(n\in\mathbb{N}\).
\end{corollary}
\begin{proof}
  This is a consequence of \Cref{thm:colimit} and the work of Fujii, Hoshino and
  Maehara~\cite[Theorems~4.3.1,~5.4.9]{fujii$o$equifibrationsStrictWeak2025}.
\end{proof}


\section{Models of \texorpdfstring{\cattinv}{ICaTT} and marked
  \texorpdfstring{\(\omega\)}{ω}-categories}
\label{sec:marked}
In \Cref{sec:cattinv-semantics}, we have studied the \(\omega\)-categories that
one may describe in \cattinv. The results of this section do
not rely on the term constructors \(\coind\), \(\can\) or \(\rec\), but only on
the presence of the destructors. In this section, we show that the above
term constructors are precisely what is needed to construct semantics in
fibrant marked \(\omega\)\-categories.

\subsection{Marked \texorpdfstring{\(\omega\)}{ω}-categories}

Marked \(\omega\)-categories have been considered in several recent works
on higher categories. Marking certain cells allows for defining finer notions of
invertibility as explained by Henry and Loubaton~\cite{loubatonInductiveModelStructure2023}. These
authors have introduced a family of model structures on marked strict
\(\omega\)\-categories with fibrant objects the marked \(\omega\)\-categories
in which marked cells are precisely these which have marked
inverses~\cite[Lemma~3.30]{loubatonInductiveModelStructure2023}. In the weak
case, marked \(\omega\)-categories were recently used to characterise
equifibrations of weak
\(\omega\)\-categories~\cite{fujii$o$equifibrationsStrictWeak2025}.

\begin{definition}
  A \emph{marked \(\omega\)-category} is an \(\omega\)-category \((X,tX)\)
  with a distinguished set of positive-dimensional cells satisfying the
  following saturation condition: for every morphism
  \(\gamma\colon \interp{\Gamma} \to X\) out of a pasting
  diagram and every full type \(\fulltyping{\Gamma}{A}\) of dimension
  \(n\) such that \(\interp{x}(\gamma)\) is marked for every
  variable \(x\in\Var_{n+1}(\Gamma)\), the cell
  \(\interp{\coh_{\Gamma,A}}(\gamma)\) is marked. We denote by
  \(\markedomegacat\) the category of marked \(\omega\)-categories and
  marking-preserving strict \(\omega\)-functors.
\end{definition}

We note that Fujii, Hoshino and Maehara~\cite{fujii$o$equifibrationsStrictWeak2025} do not require our
saturation condition, which is an analogue of the one imposed by
Henry and Loubaton~\cite{loubatonInductiveModelStructure2023}. Of course, an arbitrary collection
of cells may be closed under this saturation condition: given an
\(\omega\)\-category \(X\) and a set of cells \(S\), we define an increasing
sequence of sets by \(S_0 = S\) and
\[
  S_{n+1} = S_n \cup \left\{
    \ \interp{\coh_{\Gamma,A}}(\gamma)
    \begin{array}{l|ll}
      & \psctxty{\Gamma},
      & \fulltyping{\Gamma}{A} \text{ of dimension }n, \\
      & \gamma : \interp{\Gamma} \to X,
      & \forall
        x\in\Var_{n+1}{\Gamma},\ f(x)\in S_n
    \end{array}
  \right\}
\]
and we define \(\Sat(S)\) to be the union of this sequence. Since the operations
of \(\omega\)\-categories are finitary, this set provides the least marking on
\(X\) containing \(S\).

\begin{proposition}\label{prop:free-marked}
  The functor \(\forgetmarking\ty\markedomegacat\to\omegacat\) forgetting the
  marking has a left adjoint
  \(\freemarking{(-)}\) such that \(\forgetmarking\circ\
  \freemarking{(-)}=\id\).
\end{proposition}
\begin{proof}
  We define \(\freemarking{X} = (X, \Sat(\emptyset))\) for every
  \(\omega\)\-category \(X\), and we define
  \(\freemarking{f} = f\colon \freemarking{X}\to \freemarking{Y}\) for every
  morphism of \(\omega\)\-categories \(f\colon X\to Y\). We can show by
  induction on the definition of saturation and by functoriality of \(f\) that
  \(\freemarking{f}\) is marking-preserving. Moreover, it satisfies
  \(\forgetmarking\circ\ \freemarking{(-)}=\id\) by definition. We take as unit
  of this adjunction the identity strict \(\omega\)\-functor
  \(X\to \forgetmarking\freemarking{X}\) and as counit
  \(\freemarking{(\forgetmarking X)} \to X\) the identity of
  \(\forgetmarking X\). The counit preserves the marking by minimality of the
  saturation. The triangle identities are both satisfied, since both the unit
  and counit are given by identity functors.
\end{proof}

\begin{proposition}\label{prop:marked-cocomplete}
  The category of marked \(\omega\)-categories is complete and cocomplete.
\end{proposition}
\begin{proof}
  Limits and filtered colimits of marked \(\omega\)\-categories are computed by
  forming the limits and filtered colimits of the underlying
  \(\omega\)\-categories and sets of marked cells. These classes of marked cells
  can be shown to be saturated using that limits and filtered colimits of
  \(\omega\)\-categories are computed point-wise in the category of functors
  \(\Syn{\catt}\to \Set\). Arbitrary colimits are computed by forming the colimit
  of the underlying \(\omega\)\-categories and the saturation of the colimit of
  the marked cells.
\end{proof}

Cheng~\cite{cheng2007omega} shows that an \(\omega\)\-category where every cell
has a dual must be an \(\omega\)\-groupoid. Here, dual refers to a truncated
form of invertibility where the cancellation witnesses are not themselves
required to be invertible. Hence, that would force any definition of an
\(\omega\)\-category of cobordisms to only produce an \(\omega\)\-groupoid. This
shortcoming may be resolved by working with marked \(\omega\)\-categories, which
provide an alternative notion of
invertibility~\cite{loubatonInductiveModelStructure2023}.

\begin{definition}\label{def:marked-inv}
  A cell \(f\colon x\to y\) is called \emph{invertible up to marking} when there
  exist \emph{inverses} and \emph{cancellators}:
  \begin{align*}
    f^{\leftinv}&\ty y \to x &
    f^{\rightinv}&\ty y \to x &
    f^{\leftunit}&\ty \comp{f^{\leftinv}}{f} \to \id y &
    f^{\rightunit}&\ty \comp{f}{f^{\rightinv}} \to \id x
  \end{align*}
  such that \(f^{\leftunit}\) and \(f^{\rightunit}\) are marked.
\end{definition}

In general, there need not be any relation between marked cells and invertible
cells up to marking. We thus restrict our attention to \emph{fibrant} marked
\(\omega\)\-categories, defined in analogy to the fibrant objects of the
\emph{saturated inductive model structure}~\cite[Theorem~3.31]{
loubatonInductiveModelStructure2023}. Showing whether these are the fibrant
objects for some model structure would probably require constructing a model
structure on ordinary \(\omega\)\-categories as well.

\begin{definition}\label{def:fibrant-marked}
  A marked \(\omega\)\-category is called \emph{fibrant} when the following
  hold:
  \begin{itemize}
  \item every marked cell is invertible up to marking,
  \item all inverses of a marked cell are marked,
  \item given a marked cell \(e\ty x\to y\), if \(x\) is marked, so is \(y\),
  \item for every composable triple of cells \(f,g,h\) such that \(\comp{f}{g}\)
    and \(\comp{g}{h}\) are marked, \(f\), \(g\) and \(h\) are marked.
  \end{itemize}
\end{definition}

\noindent
Every cell that is invertible up to marking is also coinductively invertible
in fibrant marked \(\omega\)-categories. Moreover, for \catt contexts the
converse also holds:

\begin{proposition}\label{prop:fibrant-computads}
  A cell in \(\freemarking{\interp{\Gamma}}\) is marked if and
  only if it is invertible up to marking if and only if it is coinductively
  invertible. Hence, \(\freemarking{\interp{\Gamma}}\) is fibrant.
\end{proposition}
\begin{proof}
  This is an immediate corollary of \Cref{prop:invertible-cells-catt-context}.
\end{proof}

\subsection{Models of \texorpdfstring{\cattinv}{ICaTT} and marked
  \texorpdfstring{\(\omega\)}{ω}-categories}

We now provide semantics of \cattinv in fibrant marked \(\omega\)\-categories.
It takes the form of a functor \(\loc \ty \Mod(\cattinv) \to \markedomegacat\),
landing in fibrant marked \(\omega\)\-categories. We define
\(\loc X = (\syntaxincl^{*}X , t\loc X)\) to consist of the underlying category
of \(X\), in which a cell \(u\) is marked precisely when it factors through
\(\minterp{\Equiv{n+1}}\):
\[
  \begin{tikzcd}
    \minterp{\disk{n+1}} \ar[r,"u"] \ar[d,"\minterp{\equivDisk_{n+1}}"'] & X \\
    \minterp{\Equiv{n+1}} \ar[ur, dashed, "\hat{u}"']
  \end{tikzcd}
\]
We note that the factorisation is not unique on the nose, it is possible that a
more faithful semantics could be built using some notion of ``multi-marked''
\(\omega\)-categories. We define \(\loc f\ty \loc X \to \loc Y\) for a morphism
of models \(f : X \to Y\) to be
\(\syntaxincl^{*} f\colon \syntaxincl^{*}X\to \syntaxincl^{*} Y\), which
preserves marking by postcomposition.

\begin{restatable}{proposition}{locdef}\label{prop:left-adj-definition}
  \(\loc \ty \Mod(\cattinv) \to \markedomegacat\) is a well-defined functor.
\end{restatable}
\begin{proof}
  We need to show that the set of marked cells of \(\loc X\) is saturated, so
  suppose that we are given a cell \(u = \interp{\coh_{\Gamma,A}}(\gamma)\) of
  \(\loc X\) of dimension \(n\) such that \(\interp{x}(\gamma)\) is marked for
  all variables \(x\) of \(\Gamma\). By definition, this gives us, for every
  such \(x\), a commutative diagram:
  \[
    \begin{tikzcd}[column sep = huge]
      \minterp{\disk{n+1}}
      \ar[dr,"\interp{x}(\gamma)"{description}]
      \ar[d,"\minterp{\equivDisk_{n+1}}"']
      \ar[r,"\minterp{\classifiertm{x}}"]
      &  \minterp{\Gamma} \ar[d,"\gamma"]\\
      \minterp{\Equiv{n+1}} \ar[r,"e_{x}"']
      & X
    \end{tikzcd}
  \]
  By definition, \(\mu_{n+1}\) is a display map, and those are stable under
  product in every clan, thus we can define the following pullback, which can be
  concretely described as the context obtained from \(\Gamma\) by adjoining for
  every variable \(x\) of dimension \(n+1\), a variable \(\bar{e}_{x}\) of type
  \(\Inv(x)\). Noticing that \(\minterp{-}\) sends pullbacks along display maps
  to pushouts, the above family of commutative diagrams precisely shows that
  \(\gamma\) factors through the pushout:
  \[
    \begin{tikzcd}[ampersand replacement=\&, column sep = huge]
      \Gamma' \ar[r,dashed,"(\classifierequiv{\bar{e}_{x}})"]
      \ar[d,dashed,"p"', -{Triangle[open]}]
      \ar[dr,phantom,"\lrcorner"{very near start}]
      \& \prod\limits_{x\in\Var_{n+1}(\Gamma)}\Equiv{n+1}
      \ar[d,"(\equivDisk_{n+1})", -{Triangle[open]}]
      \&
      \coprod\minterp{\disk{n+1}}
      \ar[d,"\minterp{\equivDisk_{n+1}}"']
      \ar[r,"\minterp{\classifiertm{x}}"]
      \ar[dr,phantom,"\ulcorner"{very near end}]
      \&  \minterp{\Gamma} \ar[ddr,"\gamma", bend left]\ar[d,"\minterp{p}"]
      \\
      \Gamma \ar[r,"(\classifiertm{x})"']
      \& \prod\limits_{x\in\Var_{n+1}(\Gamma)} \disk{n+1}
      \&
      \coprod \minterp{\Equiv{n+1}} \ar[rrd,"e_{x}"',
      bend right=10]\ar[r,"\minterp{\classifierequiv{\bar{e}_{x}}}"']
      \& \minterp{\Gamma'}\ar[dr,"\bar{\gamma}",dashed]
      \\[-1cm]
      \&\&\& \& X
    \end{tikzcd}
  \]
  Letting \(t = \can(\coh_{\Gamma,A}[p]\{\bar{e}_{x}\})\), we have that
  \(\typing{\Gamma'}{t}[\Inv(\coh_{\Gamma,A}[i])]\) and this lets us construct
  the following commutative diagram, proving that \(u\) is marked:
  \begin{align*}
    \begin{tikzcd}[column sep = large, ampersand replacement=\&]
      \minterp{\disk{n+1}}
      \ar[r,"\minterp{\classifiertm{\coh_{\Gamma,A}[\id]}}"']
      \ar[d,"\minterp{\equivDisk_{n+1}}"']
      \ar[rrr,bend left=20, "u"]
      \& \minterp{\Gamma} \ar[r, "\minterp{p}"']
      \& \minterp{\Gamma'} \ar[r, "\bar \gamma"']
      \& X\\
      \minterp{\Equiv{n+1}}
      \ar[urr,"\minterp{\classifierequiv{t}}"', bend right=20]
    \end{tikzcd}\\[-\baselineskip]
    \tag*{\qedhere}
  \end{align*}
\end{proof}

\begin{restatable}{proposition}{locfibrant}
  \label{prop:loc-fibrant}
  The marked \(\omega\)\-category \(\loc X\) is fibrant.
\end{restatable}
\begin{proof}
  Consider a marked cell \(u\colon \minterp{\disk{n+1}} \to \syntaxincl^*X\) in
  \(\loc X\), and suppose that \(\hat{u}\colon \minterp{\Equiv{n+1}} \to
  \syntaxincl^*X\) is a
  factorisation of \(u = \hat{u}\circ \minterp{\mu_{n+1}}\). The desctuctors
  of \(\Inv\) give rise to six morphisms:
  \begin{align*}
    \minterp{\classifiertm{\leftinv(\evar_{n+1})}}
    &\ty \minterp{\disk{n+1}}\to\minterp{\Equiv{n+1}}
    & \minterp{\classifiertm{\rightinv(\evar_{n+1})}}
    & \ty \minterp{\disk{n+1}}\to\minterp{\Equiv{n+1}} \\
    \minterp{\classifiertm{\leftunit(\evar_{n+1})}}
    &\ty \minterp{\disk{n+2}}\to\minterp{\Equiv{n+1}}
    & \minterp{\classifiertm{\rightunit(\evar_{n+1})}}
    &\ty \minterp{\disk{n+2}}\to\minterp{\Equiv{n+1}} \\
    \minterp{\classifiertm{\leftunitwitness(\evar_{n+1})}}
    &\ty \minterp{\Equiv{n+2}}\to\minterp{\Equiv{n+1}}
    & \minterp{\classifierequiv{\rightunitwitness(\evar_{n+1})}}
    &\ty \minterp{\Equiv{n+2}}\to\minterp{\Equiv{n+1}}
  \end{align*}
  Composing \(\hat{u}\) wiht these morphisms above, we obtain six morphisms
  with target \(X\) which witness that \(u\) is invertible up to marking in
  \(\loc X\). Similarly, the terms constructed in \Cref{prop:left-inv},
  witnessing that inverses are invertible:
  \begin{align*}
    \typing{\Equiv{1}}{l}[\Inv(\leftinv(\evar_{1}))]
    && \typing{\Equiv{1}}{r}[\Inv(\rightinv(\evar_{1}))]
  \end{align*}
  give rise to two morphisms:
  \begin{align*}
    \minterp{\classifierequiv{\Susp^{n}l}}
    & \ty \minterp{\Equiv{n+1}}\to\minterp{\Equiv{n+1}} &
    \minterp{\classifierequiv{\Susp^{n}r}}
    & \ty \minterp{\Equiv{n+1}}\to\minterp{\Equiv{n+1}}
  \end{align*}
  Composing \(\hat{u}\) with these morphisms, we obtain two morphisms
  \(\minterp{\Equiv{n+1}}\to \syntaxincl^*X\) that witness that the left and
  right inverse of \(u\) are invertible. The remaining two fibrancy conditions
  are obtained similarly by precomposition of \(\hat{u}\) with the morphisms
  classifying the suspensions of the terms constructed in
  \Cref{prop:transport,prop:2of6}.
\end{proof}

\begin{corollary}\label{coro:marked-computad}
  For a context \(\Gamma\) of \catt, we have
  \(\loc\minterp{\Gamma} = \freemarking{\interp{\Gamma}}\)
\end{corollary}
\begin{proof}
  \Cref{coro:minterp-restriction} shows that the underlying
  \(\omega\)-categories of \(\loc\minterp{\Gamma}\) and of
  \(\freemarking{\interp{\Gamma}}\) coincide. By fibrancy, every marked cell in
  \(\loc\minterp{\Gamma}\) is coinductively invertible, hence it is marked in
  \(\freemarking{\interp{\Gamma}}\) by \Cref{prop:fibrant-computads}. Since
  \(\freemarking{(-)}\) is minimally marked, the two markings must coincide.
\end{proof}

For his duality theorem, Frey~\cite[Theorem~3.3(ii)]{jonasDualityClans2025}
identifies cocontinuous functors \(\Mod(\mathbb{T}) \to C\) to a cocomplete
category \(C\) with morphisms \(\Syn{\mathbb{T}}^{\op} \to C\) preserving
initial objects and pushouts along display maps. We believe that the functor
\[
  \loc\circ\minterp{-} \ty \Syn{\cattinv}^{\op} \to \markedomegacat
\]
satisfies those conditions, giving rise to an adjunction between
\(\Mod(\cattinv)\) and \(\markedomegacat\). This conjecture refines the
description of the expressive power of \cattinv from \(\omegacat\) to
\(\markedomegacat\). It states that in addition to having semantics in marked
\(\omega\)-categories, \cattinv{} may also be used to reason about marked
\(\omega\)-categories. Indeed, in that case, every marked \(\omega\)-category
gives rise to an \cattinv{} model.


\bibliographystyle{plainurl}
\bibliography{bibliography}
\clearpage
\appendix
\section{The syntax and rules of \texorpdfstring{\catt}{CaTT}}
\label{app:catt-rules}

We provide a complete list of the syntax and rules of the type theory \catt. We
give all the rules with named variables for the sake of readability, but a
name-free presentation is also possible, and it is commonly achieved with De
Bruijn indices. The raw syntax is given by the following grammar, where
\(\mathcal{V}\) is a countably infinite set of variables:
\[
  \begin{tabular}{|lr@{\;\;}l@{\;}l|}
    \hline
    Variable & \(x\Coloneqq\)
      & \multicolumn{2}{l|}{\(x\in\mathcal{V}\)} \\
    Type & \(A,B\Coloneqq\)
      & \(\obj\) & \(\vert\ \arr[A]{t}{u}\) \\
    Term & \(t,u\Coloneqq\)
      & \(x\) & \(\vert\ \coh_{\Gamma,A}[\gamma]\)\\
    Context & \(\Gamma,\Delta\Coloneqq\)
      & \(\emptycontext\) & \(\vert\ \Gamma \extctx (x:A)\) \\
    Substitution & \(\gamma,\delta\Coloneqq\)
      & \(\emptysub\) & \(\vert\ \gamma \extsub (x \mapsto t)\) \\
    \hline
  \end{tabular}
\]
The action of substitution on terms and types, and composition of substitutions
are defined mutually recursively by the following formulas:
\begin{align*}
  \obj[\gamma] &= \obj &
  (\arr[A]{u}{v})[\gamma] &= \arr[{A[\gamma]}]{u[\gamma]}{v[\gamma]} \\
  x[\emptysub] &= x &
  x [\gamma \extsub (y \mapsto t)] &= \begin{cases}
    t& \text{if } y = x \\
    x[\gamma] & \text{otherwise}
  \end{cases} &
  \coh_{\Gamma,A}[\delta][\gamma] &= \coh_{\Gamma,A}[\delta\circ \gamma] \\
  \emptysub\circ\ \gamma &= \emptysub &
  (\delta \extsub (x\mapsto t))\circ\gamma
    &= (\delta\circ\gamma)\extsub (x\mapsto t[\gamma])
\end{align*}
The \emph{dimension} of a type or context is given recursively by:
\begin{align*}
  \dim(\obj) &= -1 & \dim(\arr[A]{u}{v}) &= \dim A + 1 \\
  \dim(\emptycontext) &= -1 &
  \dim(\Gamma \extctx (x:A)) &= \max(\dim \Gamma, \dim A + 1)
\end{align*}
while the dimension of a typed term \(\Gamma \vdash t : A\) is defined as
\(\dim t = \dim A + 1\). The variables of a context, term, substitution or
type are defined by the following recursive equations:
\begin{align*}
  \Var(\emptycontext) &= \emptyset &
  \Var(\Gamma \extctx (x\ty A)) &= \Var(\Gamma) \cup \set{x} \\
  \Var(x) &= \set{x} &
  \Var(\coh_{\Gamma,A}[\gamma]) &= \Var(\gamma) \\
  \Var(\emptysub) &= \emptyset &
  \Var(\gamma \extsub (x\mapsto t)) &= \Var(\gamma) \cup \Var(t) \\
  \Var(\obj) &= \emptyset &
  \Var(\arr[A]{u}{v}) &= \Var(A) \cup \Var(u) \cup \Var(v)
\end{align*}
For a typed term \(\Gamma \vdash t : A\), we denote also by
\(\Var_\Gamma(t) = \Var(t) \cup \Var(A)\). We also write that
\((x\ty A)\in\Gamma\) when \(\Gamma = \Gamma'\extctx (y\ty B)\) and either
\(x = y\) and \(A = B\), or \((x\ty A)\in\Gamma'\). We say finally that a
variable \(x\in\Var(\Gamma)\) is a \emph{source (resp. target) variable} of
\(\Gamma\) when there exists no variable \((y\ty \arr[A]{u}{v})\in \Gamma\)
such that \(x\in\Var(v)\) (resp. \(x\in\Var(u)\)). This concludes the
description of the raw syntax of \catt.

The judgements of \catt are defined mutually recursively by a set of rules.
The rules for contexts, variables, and substitutions are the usual ones for
dependent type theory:
\begin{mathpar}
  \inferrule{
    \null
  }{
    \ctxty{\emptycontext}
  }\and
  \inferrule{
    \typing{\Gamma}{A} \\ x\notin \Var\Gamma
  }{
    \ctxty{\Gamma\extctx (x \ty A)}
  }\and
  \inferrule{
    \typing{\Gamma} \\ (x\ty A)\in\Var\Gamma
  }{
    \ctxty{\typing{\Gamma}{x}[A]}
  } \\
  \inferrule{
    \ctxty{\Gamma}
  }{
    \typing{\Gamma}{\emptysub}[\emptycontext]
  }\and
  \inferrule{
    \typing{\Delta}{\gamma}[\Gamma]\\
    \typing{\Gamma}{A} \\
    \typing{\Delta}{t}[A[\gamma]]
  }{
    \typing{\Delta}{\gamma\extsub (x\mapsto t)}[\Gamma\extctx (x\ty A)]
  }
\end{mathpar}
\catt has a single base type \(\obj\) and a type former \(\!\arr{}{}\!\) for
higher cells:
\begin{mathpar}
  \inferdef{\ctxty{\Gamma}}{\typing{\Gamma}{\obj}}
  \and
  \inferdef{
    \typing{\Gamma}{u}[\obj] \\
    \typing{\Gamma}{v}[\obj] }
  {\typing{\Gamma}{\arr[\obj]{u}{v}}}
  \and
  \inferdef{
    \typing{\Gamma}{u}[\arr[A]{v}{w}] \\
    \typing{\Gamma}{u'}[\arr[A]{v}{w}] }
  {\typing{\Gamma}{\arr[{\arr[A]{v}{w}}]{u}{u'}}}
\end{mathpar}
To introduce the typing rule for coherence terms, we first identify a class of
contexts called \emph{pasting contexts} via the following four rules:
\begin{mathpar}
  \inferrule
  { }
  {\pstyping{(x\ty\obj)}{x}[\obj]}
  \and
  \inferrule
  {\pstyping{\Gamma}{f}[\arr[A]{x}{y}]}
  {\pstyping{\Gamma}{y}[A]}
  \\
  \inferrule
  {\pstyping{\Gamma}{x}[A]}
  {\pstyping{\Gamma\extctx(y\ty A)\extctx(f\ty \arr[A]{x}{y})}{f}[{\arr[A]{x}{y}}]}
  \and
  \inferrule
  {\pstyping{\Gamma}{x}[\obj]}
  {\psctxty{\Gamma}}
\end{mathpar}
We also define an auxiliary judgement \(\fulltyping{\Gamma}{A}\) for a type
\(\typing{\Gamma}{A}\) over a pasting context \(\psctxty{\Gamma}\) to hold when
\(A = \arr[B]{u}{v}\) is an arrow type such that:
\begin{itemize}
  \item \(\Var_\Gamma(u) = \Var_\Gamma(v) = \Var(\Gamma)\), or
  \item \(\dim A = \dim \Gamma - 1\) and \(\Var_\Gamma(u)\) and
    \(\Var_\Gamma(v)\) contain the source and target variables of \(\Gamma\) of
    dimension at most \(\dim A\) respectively.
\end{itemize}
Using these auxiliary judgements, we can now describe the typing rule for
coherence terms:
\begin{mathpar}
  \inferdef{
    \psctxty{\Gamma} \\
    \fulltyping{\Gamma}{A} \\
    \typing{\Delta}{\gamma}[\Gamma]
  }{
    \typing{\Delta}{\coh_{\Gamma,A}[\gamma]}[A[\gamma]]
  }
\end{mathpar}


\section{Invertible terms in \texorpdfstring{\catt}{CaTT}}
\label{app:inverses}

This appendix is dedicated to a more detailed presentation of the construction
of \Cref{thm:invertibility-catt} of Benjamin and
Markakis~\cite{benjaminInvertibleCells$omega$categories2024}, on which the
\(\beta\)\-reduction rules for the \(\can\) constructor of \cattinv depend. To
sketch the idea of the theorem, we recall the \emph{suspension} operation in
\catt defined on contexts, types, terms, and substitutions by the following
recursive equations, where \(v^-\) and \(v^+\) are assumed to be distinct,
fresh variables:
\begin{align*}
  \Susp (\emptycontext)
    &= (v^-\ty \obj) \extctx (v^+\ty \obj) &
  \Susp (\Gamma\extctx (x \colon A))
    &= (\Susp \Gamma)\extctx (\Susp x \colon \Susp A) \\
  \Susp (\obj)
    &= \arr[A]{v^-}{v^+} &
  \Susp (\arr[A]{u}{v})
    &= \arr[\Susp A]{\Susp u}{\Susp v} \\
  \Susp (x)
    &= x &
  \Susp (\coh_{\Gamma,A}[\gamma])
    &=  \coh_{\Susp \Gamma,\Susp A}[\Susp \gamma] \\
  \Susp(\emptysub)
    &= (v^-\mapsto v^-) \extsub (v^+\mapsto v^+) &
  \Susp(\gamma\extsub (x\mapsto t))
    &= (\Susp \gamma)\extsub (\Susp x \mapsto \Susp t)
\end{align*}
Additionally, we require the meta-operation corresponding to the formation of
opposite \(\omega\)\-categories~\cite{benjaminHom$omega$categoriesComputad2024}.
Given an \(n\ge 1\), this operation swaps the source and target of terms of
dimension \(n\in\N\) as follows:
\begin{align*}
  \op_{n}\emptycontext
  &= \emptycontext
  & \op_{n} \Gamma\extctx(x\ty A)
  &= (\op_{n}\Gamma) \extctx (x\ty \op_{n}A) \\
  \op_{n}\obj
  &= \obj
  & \op_{n} (\arr[A]{u}{v})
  &=
    \begin{cases}
      \arr{\op_{n}u}{\op_{n} v}\!&\! (\dim A+2 \neq n) \\
      \arr{\op_{n}v}{\op_{n} u}\!&\! (\dim A+2 = n)
    \end{cases}\\
  \op_{n}x
  &= x
  & \op_{n}\coh_{\Gamma,A}[\gamma]
  &= \coh_{\overline{\op_{n}\Gamma},(\op_{n}A)[
    \op_{n}^{\Gamma,-1}]}[\op_{n}^{\Gamma}\circ\op_{n}\gamma] \\
  \op_{n}\emptysub
  &= \emptysub
  & \op_{n}\gamma\extsub(x\mapsto t)
  &= (\op_{n}\gamma)\extsub(x\mapsto \op_{n}t)
\end{align*}
where \(\typing{\op_{n}\Gamma}{\op_{n}^{\Gamma}}[\overline{\op_n\Gamma}]\) is
the unique isomorphism between \(\op_n\Gamma\) and a pasting diagram. This
meta-operation preserves valid contexts, terms, types and substitutions, and it
preserves pasting diagram and fullness up to unique isomorphism. A key lemma
for the construction of inverses of composite terms is the following:

\begin{lemma}
  \cite[Lemma~31]{benjaminInvertibleCells$omega$categories2024}
  For every pasting diagram \(\pstyping{\Gamma}\) of dimension at most \(n\),
  the \(n\)\-opposite pasting diagram \(\overline{\op_n\Gamma}\) is equal to
  \(\Gamma\). Moreover, given a morphism \(\gamma\colon\interp{\Gamma} \to X\)
  and invertibility structures on \(\interp{x}(\gamma)\) for every
  \(n\)\-dimensional variable \(x\) of \(\Gamma\),
  there exist two \(\omega\)\-functors
  \begin{align*}
    \gamma^{\leftinv} &\ty \interp{\Gamma} \to X &
    \gamma^{\rightinv} & \ty \interp{\Gamma} \to X
  \end{align*}
  determined by agreeing with \(\gamma\) on cell of dimension less than \(n\)
  and sending each \(n\)\-dimensional variable \(x\) to
  \begin{align*}
    \interp{x}(\gamma^{\leftinv})
      &= x[\op_{n}^{\Gamma}]^{\leftinv} &
    \interp{x}(\gamma^{\rightinv})
      &= x[\op_{n}^{\Gamma}]^{\rightinv}
  \end{align*}
\end{lemma}

We remark first that if the pasting diagram \(\Gamma\) is of dimension strictly
less than \(n\), there exist no \(n\)\-dimensional variables in \(\Gamma\), so
\(\op_n\Gamma = \Gamma\) and
\(\gamma^{\leftinv} = \gamma^{\rightinv} = \gamma\). The case of interest is
when \(\Gamma\) is of dimension \(n\), in which case showing that
\(\overline{\op_n\Gamma} = \Gamma\) can be seen using the correspondence between
pasting diagrams and trees~\cite{bataninMonoidalGlobularCategories1998}. We
illustrate the lemma in the following two low-dimensional examples:
\begin{align*}
  \Gamma &=
  \begin{tikzcd}[ampersand replacement=\&]
    x \ar[r, "f"]
    \& y \ar[r, "g"]
    \& z
  \end{tikzcd} &
  \op_1\Gamma &=
  \begin{tikzcd}[ampersand replacement=\&]
    z \ar[r, "g"]
    \& y \ar[r, "f"]
    \& x
  \end{tikzcd} \\
  \Delta &=
  \begin{tikzcd}[ampersand replacement=\&]
    x \& y \& z
    \arrow[""{name=0, anchor=center, inner sep=0},
      "f", shift left, curve={height=-12pt}, from=1-1, to=1-2]
    \arrow[""{name=1, anchor=center, inner sep=0},
      "h"', shift right, curve={height=12pt}, from=1-1, to=1-2]
    \arrow[""{name=2, anchor=center, inner sep=0},
      "g"{description}, from=1-1, to=1-2]
    \arrow["k", from=1-2, to=1-3]
    \arrow["a", between={0.2}{0.8}, Rightarrow, from=0, to=2]
    \arrow["b", between={0.2}{0.8}, Rightarrow, from=2, to=1]
  \end{tikzcd} &
  \op_2\Delta &=
  \begin{tikzcd}[ampersand replacement=\&]
    x \& y \& z
    \arrow[""{name=0, anchor=center, inner sep=0},
      "h", shift left, curve={height=-12pt}, from=1-1, to=1-2]
    \arrow[""{name=1, anchor=center, inner sep=0},
      "f"', shift right, curve={height=12pt}, from=1-1, to=1-2]
    \arrow[""{name=2, anchor=center, inner sep=0},
      "g"{description}, from=1-1, to=1-2]
    \arrow["k", from=1-2, to=1-3]
    \arrow["b", between={0.2}{0.8}, Rightarrow, from=0, to=2]
    \arrow["a", between={0.2}{0.8}, Rightarrow, from=2, to=1]
  \end{tikzcd}
\end{align*}
The morphisms \(\gamma^{\leftinv}\) and \(\gamma^{\rightinv}\) in these examples
can be represented graphically as follows:
\begin{align*}
  \gamma^{\leftinv} &=
  \begin{tikzcd}[ampersand replacement=\&, column sep = small]
    \interp{z}(\gamma) \ar[r, "{g^{\leftinv}}"]
    \& \interp{y}(\gamma) \ar[r, "{f^{\leftinv}}"]
    \& \interp{x}(\gamma)
  \end{tikzcd}
  &
  \gamma^{\leftinv} &=
  \begin{tikzcd}[ampersand replacement = \&, column sep = large]
    \interp{x}(\gamma) \& \interp{y}(\gamma) \& \interp{z}(\gamma)
    \arrow[""{name=0, anchor=center, inner sep=0},
      "\interp{h}(\gamma)", shift left, curve={height=-15pt}, from=1-1, to=1-2]
    \arrow[""{name=1, anchor=center, inner sep=0},
      "\interp{f}(\gamma)"', shift right, curve={height=15pt}, from=1-1, to=1-2]
    \arrow[""{name=2, anchor=center, inner sep=0},
      "\interp{g}(\gamma)"{description}, from=1-1, to=1-2]
    \arrow["\interp{k}(\gamma)", from=1-2, to=1-3]
    \arrow["{b^{\leftinv}}", between={0.2}{0.8}, Rightarrow, from=0, to=2]
    \arrow["{a^{\leftinv}}", between={0.2}{0.8}, Rightarrow, from=2, to=1]
  \end{tikzcd}
  \\
  \gamma^{\rightinv} &=
  \begin{tikzcd}[ampersand replacement=\&, column sep = small]
    \interp{z}(\gamma) \ar[r, "{g^{\rightinv}}"]
    \& \interp{y}(\gamma) \ar[r, "{f^{\rightinv}}"]
    \& \interp{x}(\gamma)
  \end{tikzcd}
  &
  \gamma^{\rightinv} &=
  \begin{tikzcd}[ampersand replacement = \&, column sep = large]
    \interp{x}(\gamma) \& \interp{y}(\gamma) \& \interp{z}(\gamma)
    \arrow[""{name=0, anchor=center, inner sep=0},
      "\interp{h}(\gamma)", shift left, curve={height=-15pt}, from=1-1, to=1-2]
    \arrow[""{name=1, anchor=center, inner sep=0},
      "\interp{f}(\gamma)"', shift right, curve={height=15pt}, from=1-1, to=1-2]
    \arrow[""{name=2, anchor=center, inner sep=0},
      "\interp{g}(\gamma)"{description}, from=1-1, to=1-2]
    \arrow["\interp{k}(\gamma)", from=1-2, to=1-3]
    \arrow["{b^{\rightinv}}", between={0.2}{0.8}, Rightarrow, from=0, to=2]
    \arrow["{a^{\rightinv}}", between={0.2}{0.8}, Rightarrow, from=2, to=1]
  \end{tikzcd}
\end{align*}
In general, this construction replaces the \(n\)\-dimensional arguments by their
inverses, and reverses the order in which they appear.

\thminvertibility*

\begin{proof}[Proof Sketch]
  The theorem is proven by coinduction, providing explicit left and right
  inverse and cancellators. The cancellators are cells of the same form as the
  original cell, so they are invertible by coinduction. Considering a cell of
  the form \(t = \interp{\coh_{\Gamma,A}}(\gamma)\) as in the statement and
  writing \(A = \arr{u}{v}\), we define
  \begin{align*}
    t^{\leftinv} &= \interp{\op_{n+1} \coh_{\Gamma,A}}(\gamma^{\leftinv}) &
    t^{\rightinv} &= \interp{\op_{n+1} \coh_{\Gamma,A}}(\gamma^{\rightinv})
  \end{align*}
  The cancellation witnesses in the case that \(\dim\Gamma\le n\) are
  of the form \(\interp{\coh_{\Gamma,A'}}(\gamma)\) for where
  \begin{align*}
    A' = \begin{cases}
      \arr{\op_{n+1}\coh_{\Gamma,A}[\id] * \coh_{\Gamma,A}[\id]}{\id v}
        & \text{for the left cancellator} \\
      \arr{\coh_{\Gamma,A}[\id] * \op_{n+1}\coh_{\Gamma,A}[\id]}{\id u}
        & \text{for the right cancellator}
    \end{cases}
  \end{align*}
  In the case that \(\dim\Gamma = n+1\), the cancellation witnesses are more
  involved, and they are constructed in three steps. The first step is a
  generalised associator that reassociates \(t^{\leftinv}\ast t\) or
  \(t * t^{\rightinv}\) in order to to group together \emph{chains} of
  \((n+1)\)\-cells with their corresponding inverses. The second step consists
  in cancelling these chains of cells into identities, while the final step
  consists in a generalised unitor that relates a composite of identities to a
  single identity. We illustrate these steps for the left cancellator in
  \Cref{fig:cancellation-witness}.
\end{proof}

\begin{figure}
  \begin{align*}
    \begin{tikzcd}[ampersand replacement = \&, column sep = huge, row sep = small]
      {\interp{x}[\gamma]} \& {\interp{y}[\gamma]} \&[-.5cm] {\interp{z}[\gamma]}
      \\[-.2cm]
      \& \comp{}{}
      \\[-.2cm]
      {\interp{x}[\gamma]} \& {\interp{y}[\gamma]} \&[-.5cm] {\interp{z}[\gamma]}
      \arrow[""{name=0, anchor=center, inner sep=0},
        "{\interp{h}[\gamma]}", curve={height=-60pt}, from=1-1, to=1-2]
      \arrow[""{name=1, anchor=center, inner sep=0},
        "{\interp{f}[\gamma]}"', from=1-1, to=1-2]
      \arrow[""{name=2, anchor=center, inner sep=0},
        "{\interp{g}[\gamma]}"{description},
        shift left=2, curve={height=-12pt}, from=1-1, to=1-2]
      \arrow["{\interp{k}[\gamma]}", from=1-2, to=1-3]
      \arrow[""{name=3, anchor=center, inner sep=0},
        "{\interp{f}[\gamma]}", from=3-1, to=3-2]
      \arrow[""{name=4, anchor=center, inner sep=0},
        "{\interp{h}[\gamma]}"', curve={height=60pt}, from=3-1, to=3-2]
      \arrow[""{name=5, anchor=center, inner sep=0},
        "{\interp{g}[\gamma]}"{description}, shift right=2,
        curve={height=12pt}, from=3-1, to=3-2]
      \arrow["{\interp{k}[\gamma]}", from=3-2, to=3-3]
      \arrow["{b^{\leftinv}}", between={0.2}{0.8},
        Rightarrow, from=0, to=2]
      \arrow["{a^{\leftinv}}", between={0.2}{0.8},
        Rightarrow, from=2, to=1]
      \arrow["{\interp{a}[\gamma]}", between={0.2}{0.8},
        Rightarrow, from=3, to=5]
      \arrow["{\interp{b}[\gamma]}", between={0.2}{0.8},
        Rightarrow, from=5, to=4]
    \end{tikzcd} &
    \begin{tikzcd}[ampersand replacement = \&]
      \\
      \\
      \&\phantom{} \\
      \\
      \&\&\&\& \phantom{}
      \arrow["\operatorname{assoc}",
        Rightarrow, scaling nfold=3, bend left = 30, from=3-2, to=5-5]
    \end{tikzcd}
    \\[-1cm]
    \begin{tikzcd}[ampersand replacement = \&]
      \\
      \\
      \&\&\&\&\&\&\phantom{}\& \\
      \&\& \phantom{}
      \arrow["\operatorname{cancel}",
        Rightarrow, scaling nfold=3, bend right = 30, from=3-7, to=4-3]
    \end{tikzcd} &
    \begin{tikzcd}[ampersand replacement = \&, column sep = huge, row sep = small]
      {\interp{x}[\gamma]} \& {\interp{y}[\gamma]} \&[-.5cm] {\interp{z}[\gamma]}
      \arrow[""{name=6, anchor=center, inner sep=0},
        "{\interp{h}[\gamma]}", curve={height=-30pt}, from=1-1, to=1-2]
      \arrow[""{name=7, anchor=center, inner sep=0},
        "{\interp{h}[\gamma]}"', curve={height=30pt}, from=1-1, to=1-2]
      \arrow["{\interp{k}[\gamma]}", from=1-2, to=1-3]
      \arrow["{C}"{description}, Rightarrow, between={0.2}{0.8}, from=6, to=7]
    \end{tikzcd}
    \\
    \begin{tikzcd}[ampersand replacement = \&, column sep = huge, row sep = small]
      {\interp{x}[\gamma]} \& {\interp{y}[\gamma]} \&[-.5cm] {\interp{z}[\gamma]}
      \arrow[""{name=6, anchor=center, inner sep=0},
        "{\interp{h}[\gamma]}", curve={height=-30pt}, from=1-1, to=1-2]
      \arrow[""{name=7, anchor=center, inner sep=0},
        "{\interp{h}[\gamma]}"', curve={height=30pt}, from=1-1, to=1-2]
      \arrow["{\interp{k}[\gamma]}", from=1-2, to=1-3]
      \arrow["{\id \interp{h}(\gamma)}"{description}, Rightarrow,
      between={0.2}{0.8}, from=6, to=7]
    \end{tikzcd} &
    \begin{tikzcd}[ampersand replacement = \&]
      \&\phantom{} \\
      \&\&\&\& \phantom{}
      \arrow["\operatorname{unitor}",
        Rightarrow, scaling nfold=3, bend left = 30, from=1-2, to=2-5]
    \end{tikzcd}
    \\[-.7cm] &
    \begin{tikzcd}[ampersand replacement = \&, column sep = huge, row sep = small]
      {\interp{x}[\gamma]} \& {\phantom{\interp{y}[\gamma]}}
      \&[-.5cm] {\interp{z}[\gamma]}
      \arrow[""{name=6, anchor=center, inner sep=0},
        "{\interp{\comp{h}{k}}[\gamma]}", curve={height=-30pt}, from=1-1, to=1-3]
      \arrow[""{name=7, anchor=center, inner sep=0},
        "{\interp{\comp{h}{k}}[\gamma]}"', curve={height=30pt}, from=1-1, to=1-3]
      \arrow["{\id \interp{\comp{h}{k}}(\gamma)}"{description},
        Rightarrow, between={0.2}{0.8}, from=6, to=7]
    \end{tikzcd}
  \end{align*}
  \caption{The left cancellator of a composite of
    invertible cells consists of an associator, a cancellation, and a unitor.
    Here \(C = \comp{(\comp{b^{\leftinv}}{a^{\leftinv}})}
      {(\comp{\interp{a}[\gamma]}{\interp{b}[\gamma]})}\).}
    \label{fig:cancellation-witness}
\end{figure}

\section{The syntax and rules of \texorpdfstring{\cattinv}{CaTTInv}}
\label{app:cattinv}

In this section of the appendix, we give a complete description of the syntax
and rules of the type theory \cattinv. Its raw syntax extends that of \catt with
a type constructor \(\Inv_A(t)\) where \(A\) is a type and \(t\) is a term and
with nine term constructors listed below, where \(\gamma\) is a substitution,
\(\{e_{x}\}\) is a family of terms indexed by a finite set, and the remaining
arguments are terms:
\begin{mathparpagebreakable}
  \can(t,\{e_{x}\}) \and
  \coind(t,t_{l},t_{r},t_{lu},t_{ru},t_{ilu},t_{iru}) \and
  \rec(t,t_{l},t_{r},t_{lu},t_{ru},t_{ilu},t_{iru},\gamma) \\
  \leftinv(t)
  \and \rightinv(t)
  \and \leftunit(t)
  \and \rightunit(t)
  \and \leftunitwitness(t)
  \and \rightunitwitness(t)
\end{mathparpagebreakable}
The action of substitutions on terms and types, as well as composition of
substitutions are defined by the same recursive equations as in \catt together
with:
\begin{align*}
  \Inv_{A}(t)[\sigma] &= \Inv_{A[\sigma]}(t[\sigma]) &
  \can(t,\{e_{x}\})[\sigma] &= \can(t[\sigma],\{e_{x}[\sigma]\}) \\
  \leftinv(t)[\sigma] &= \leftinv(t[\sigma]) &
  \rightinv(t)[\sigma] &= \rightinv(t[\sigma]) \\
  \leftunit(t)[\sigma] &= \leftunit(t[\sigma]) &
  \rightunit(t)[\sigma] &= \rightunit(t[\sigma]) \\
  \leftunitwitness(t)[\sigma] &= \leftunitwitness(t[\sigma]) &
  \rightunitwitness(t)[\sigma] &= \rightunitwitness(t[\sigma]) \\
  \omit\rlap{\(\coind(t,t_{l},t_{r},t_{lu},t_{ru},t_{ilu},t_{iru})[\sigma]
    = \coind(t[\sigma],t_{l}[\sigma],t_{r}[\sigma],t_{lu}[\sigma],
    t_{ru}[\sigma],t_{ilu}[\sigma],t_{iru}[\sigma])\)} \\
  \omit\rlap{\(\rec(t,t_{l},t_{r},t_{lu},t_{ru},t_{ilu},t_{iru},\gamma)[\sigma]
    = \rec(t,t_{l},t_{r},t_{lu},t_{ru},t_{ilu},t_{iru},\gamma\circ\,\sigma)\)}
\end{align*}
Similarly, the suspension operation on contexts, substitutions, types and
terms is defined by the same recursive equations as in \catt together with:
\begin{align*}
  \Susp(\Inv_{A}(t)) &= \Inv_{\Susp A}(\Susp t) &
  \Susp(\can(t,\{e_{x}\})) &= \can(\Susp t, \{ \Susp e_{x}\}) \\
  \Susp(\leftinv(e)) &= \leftinv(\Susp e) &
  \Susp(\rightinv(e)) &= \rightinv(\Susp e) \\
  \Susp(\leftunit(e)) &= \leftunit(\Susp e) &
  \Susp(\rightunit(e)) &= \rightunit(\Susp e) \\
  \Susp(\leftunitwitness(e)) &= \leftunitwitness(\Susp e) &
  \Susp(\rightunitwitness(e)) &= \rightunitwitness(\Susp e) \\
  \omit\rlap{\(\Susp\coind(t,t_{l},t_{r},t_{lu},t_{ru},t_{ilu},t_{iru}) =
    \coind(\Susp t,\Susp t_{l}, \Susp t_{r}, \Susp t_{lu},
    \Susp t_{ru} , \Susp t_{ilu}, \Susp t_{iru})\)} \\
  \omit\rlap{\(\Susp\rec(t,t_{l},t_{r},t_{lu},t_{ru},t_{ilu},t_{iru},\gamma) =
    \rec(\Susp t,\Susp t_{l}, \Susp t_{r}, \Susp t_{lu},
    \Susp t_{ru} , \Susp t_{ilu}, \Susp t_{iru}, \Susp\gamma)\)}
\end{align*}
The type formation rules for \cattinv are the same as these of \catt extended by
the new type formation rule for \(\Inv\):
\begin{mathpar}
\inferrule{
    \typing{\Gamma}{t}[\arr[A]{u}{v}]
  }{
    \typing{\Gamma}{\Inv_{\arr[A]{u}{v}}(t)}
  }
\end{mathpar}
The rules for the six destructors, as well as the two constructors \(\coind\)
and \(\can\), can then be stated immediately:
\begin{mathparpagebreakable}
  \inferrule{
    \typing{\Gamma}{t}[\arr[A]{u}{v}]\\
    \typing{\Gamma}{e}[\Inv(t)]
  }{
    \typing{\Gamma}{\leftinv(e)}[\arr[A]{v}{u}]
  } \and
  \inferrule{
    \typing{\Gamma}{t}[\arr[A]{u}{v}]\\
    \typing{\Gamma}{e}[\Inv(t)]
  }{
    \typing{\Gamma}{\rightinv(e)}[\arr[A]{v}{u}]
  } \\
  \inferrule{
    \typing{\Gamma}{t}[\arr[A]{u}{v}]\\
    \typing{\Gamma}{e}[\Inv(t)]
  }{
    \typing{\Gamma}{\leftunit(e)}[\arr{\comp{\leftinv(e)}{t}}{\id v}]
  } \and
  \inferrule{
    \typing{\Gamma}{t}[\arr[A]{u}{v}]\\
    \typing{\Gamma}{e}[\Inv(t)]
  }{
    \typing{\Gamma}{\rightunit(e)}[\arr{\comp{t}{\rightinv(e)}}{\id u}]
  } \\
  \inferrule{
    \typing{\Gamma}{t}[\arr[A]{u}{v}]\\
    \typing{\Gamma}{e}[\Inv(t)]
  }{
    \typing{\Gamma}{\leftunitwitness(e)}[\Inv(\leftunit(e))]
  } \and
  \inferrule{
    \typing{\Gamma}{t}[\arr[A]{u}{v}]\\
    \typing{\Gamma}{e}[\Inv(t)]
  }{
    \typing{\Gamma}{\rightunitwitness(e)}[\Inv(\leftunit(e))]
  } \\
  \inferrule{
    \typing{\Gamma}{t}[\arr{x}{y}]\\
    \typing{\Gamma}{t_l}[\arr{y}{x}]\\
    \typing{\Gamma}{t_r}[\arr{y}{x}]\\\\
    \typing{\Gamma}{t_{lu}}[\arr{\comp{t_l}{t}}{\id y}] \\
    \typing{\Gamma}{t_{ru}}[\arr{\comp{t}{t_r}}{\id x}] \\\\
    \typing{\Gamma}{t_{ilu}}[\Inv(t_{lu})] \\
    \typing{\Gamma}{t_{iru}}[\Inv(t_{ru})]
  }{
    \typing{\Gamma}{\coind(t,t_{l},t_{r},t_{lu},t_{ru},t_{ilu},t_{iru})}[\Inv(t)]
  } \\
  \inferdef{
    \typing{\Delta}{\coh_{\Gamma,A}[\gamma]}[A[\gamma]] \\
    \set{\typing{\Delta}{e_x}[\Inv(x[\gamma])]}_{x\in \Var_{\dim A + 1}(\Gamma)}
  }{
    \typing{\Delta}{\can(\coh_{\Gamma,A}[\gamma],
      \set{e_x})}[\Inv(\coh_{\Gamma,A}[\gamma])]
  }
\end{mathparpagebreakable}
The rule for the final constructor \(\rec\) requires the definition
of the walking equivalence \(\Equiv{n+1}\) and the characteristic substitutions
of \Cref{lemma:familial-representability}. Letting \(\dvar_{n+1}\) be the top
dimensional variable of \(\disk{n+1}\), the walking equivalence is given by:
\[
  \Equiv{n+1} = \disk{n+1}\extctx(\evar_{n+1}\ty\Inv(\dvar_{n+1}))
\]
The substitutions \(\classifierty{A}\) are defined for a type \(A\) recursively
by:
\begin{align*}
  \classifierty{\obj}
    &= \emptysub &
  \classifierty{\Inv_{A}(t)}
    &= \classifierty{A} \extsub (\dvar_{n+1}\mapsto t) &
  \classifierty{\arr[A]{u}{v}}
    & = \classifierty{A}\extsub (\dvar^-_n\mapsto u)\extsub(\dvar^+_n\mapsto v)
\end{align*}
where \(\dvar^-_n\) and \(\dvar^+_n\) are the two top dimensional variables of
\(\sphere{n}\). Similarly, we can define for a pair of a term \(t\) and a type
\(A\) the substitution \(\classifiertm{t,A}\) recursively by:
\begin{equation}
    \classifiertm{t,A}
      = \begin{cases}
        \classifierty{A} \extsub (\dvar_{n+1}\mapsto t)
          &\text{when } A = \obj \text{ or } A = \arr{u}{v} \\
        \classifierty{A} \extsub (\evar_{n+1}\mapsto t)
          & \text{when } A = \Inv(u)
      \end{cases}
\end{equation}
We can then define the auxiliary context
\(\EquivInd{n+1}{t}\) for a term \(\typing{\Equiv{n+1}}{t}[\arr[A]{u}{v}]\) as
the context extension:
\[
  \EquivInd{n+1}{t} = \Equiv{n+1} \extctx
  (h^-\ty \Inv(\Susp t)[\classifierequiv{\leftunitwitness(\evar_{n+1})}])\extctx
  (h^+\ty \Inv(\Susp t)[\classifierequiv{\rightunitwitness(\evar_{n+1})}])
\]
The rule for the constructor \(\rec\) is given by:
\begin{mathpar}
  \inferrule{
    \typing{\Equiv{n+1}}{t}[\arr{x}{y}]\\
    \typing{\Gamma}{\gamma}[\Equiv{n+1}] \\
    \typing{\Equiv{n+1}}{t_l}[\arr{y}{x}]\\
    \typing{\Equiv{n+1}}{t_r}[\arr{y}{x}]\\\\
    \typing{\Equiv{n+1}}{t_{lu}}[\arr{\comp{t_l}{t}}{\id y}] \\
    \typing{\Equiv{n+1}}{t_{ru}}[\arr{\comp{t}{t_r}}{\id x}] \\\\
    \typing{\EquivInd{n+1}{t}}{t_{ilu}}[\Inv(t_{lu})] \\
    \typing{\EquivInd{n+1}{t}}{t_{iru}}[\Inv(t_{ru})] \\
    }{
    \typing{\Gamma}{\rec(t,t_l,t_r,t_{lu},t_{ru},t_{ilu},t_{iru},\gamma)}[
      \Inv(t[\gamma])]
  }
\end{mathpar}
Unlike \catt, the type theory \cattinv also has non-trivial judgemental
equalities between terms, types, contexts and substitutions. Judgemental
equality is the least congruence relation on the syntax of \cattinv containing
the following equalities. We first postulate \(\beta\)\- and \(\eta\)\-rules
between the destructors and the constructors \(\coind\):
\begin{mathparpagebreakable}
  \inferdef[\(\beta\)-\(\leftinv\)]{
    \typing{\Gamma}{\coind(t,t_l,t_r,t_{lu},t_{ru},t_{ilu},
      t_{iru})}[\Inv_{\arr[A]{u}{v}}(t)]
  }{
    \typing{\Gamma}{\leftinv(\coind(t,t_l,t_r,t_{lu},t_{ru},t_{ilu},t_{iru}))
      \equiv t_l}[\arr[A]{v}{u}]
  } \and
  \inferdef[\(\beta\)-\(\rightinv\)]{
    \typing{\Gamma}{\coind(t,t_l,t_r,t_{lu},t_{ru},t_{ilu},
      t_{iru})}[\Inv_{\arr[A]{u}{v}}(t)]
  }{
    \typing{\Gamma}{\rightinv(\coind(t,t_l,t_r,t_{lu},t_{ru},t_{ilu},t_{iru}))
      \equiv t_r}[\arr[A]{v}{u}]
  } \and
  \inferdef[\(\beta\)-\(\leftunit\)]{
    \typing{\Gamma}{\coind(t,t_l,t_r,t_{lu},t_{ru},t_{ilu},
      t_{iru})}[\Inv_{\arr[A]{u}{v}}(t)]
  }{
    \typing{\Gamma}{\leftunit(\coind(t,t_l,t_r,t_{lu},t_{ru},t_{ilu},t_{iru}))
      \equiv t_{lu}}[\arr{t_l\ast t}{\id v}]
  } \and
  \inferdef[\(\beta\)-\(\rightunit\)]
  {\typing{\Gamma}{\coind(t,t_l,t_r,t_{lu},t_{ru},t_{ilu},
    t_{iru})}[\Inv_{\arr[A]{u}{v}}(t)]}
  {
    \typing{\Gamma}{\rightunit(\coind(t,t_l,t_r,t_{lu},t_{ru},t_{ilu},t_{iru}))
      \equiv t_{ru}}[\arr{t \ast t_r}{\id u}]
  } \and
  \inferdef[\(\beta\)-\(\leftunitwitness\)]{
    \typing{\Gamma}{\coind(t,t_l,t_r,t_{lu},t_{ru},t_{ilu},
      t_{iru})}[\Inv_{\arr[A]{u}{v}}(t)]
  }{
    \typing{\Gamma}{\leftunitwitness(\coind(t,t_l,t_r,t_{lu},t_{ru},
      t_{ilu},t_{iru}))
      \equiv t_{ilu}}[\Inv(t_{lu})]
  } \and
  \inferdef[\(\beta\)-\(\rightunitwitness\)]{
    \typing{\Gamma}{\coind(t,t_l,t_r,t_{lu},t_{ru},t_{ilu},
      t_{iru})}[\Inv_{\arr[A]{u}{v}}(t)]
  }{
    \typing{\Gamma}{\rightunitwitness(\coind(t,t_l,t_r,t_{lu},t_{ru},
      t_{ilu},t_{iru}))
      \equiv t_{iru}}[\Inv(t_{ru})]
  } \and
  \inferdef[\(\eta\)-\(\Inv\)]{
    \typing{\Gamma}{e}[\Inv(t)]
  }{
    \typing{\Gamma} {e \equiv
      \coind(t,\leftinv(e),\rightinv(e),\leftunit(e),\rightunit(e),
      \leftunitwitness(e),\rightunitwitness(e))}[\Inv(t)]
  }
\end{mathparpagebreakable}
We then postulate computation rules for the constructor \(\can\), identifying
the invertibility structure obtained by the rule with the one obtained by
\Cref{thm:invertibility-catt}:
\begin{mathpar}
  \inferdef[\(\beta_{\can}\)-\(\leftinv\)]{
    \typing{\Delta}{\can(t,\{e_x\}_{x\in X})}[\Inv_{\arr{u}{v}}(t)]
  }{
    \typing{\Delta}{ \leftinv(\can(t,\set{e_x}))
      \equiv t^{\leftinv}}[\arr{v}{u}]
  }
  \and
  \inferdef[\(\beta_{\can}\)-\(\rightinv\)]{
    \typing{\Delta}{\can(t,\{e_x\}_{x\in X})}[\Inv_{\arr{u}{v}}(t)]
  }{
    \typing{\Delta}{ \rightinv(\can(t,\set{e_x}))
      \equiv t^{\rightinv}}[\arr{v}{u}]
  }
  \and
  \inferdef[\(\beta_{\can}\)-\(\leftunit\)]{
    \typing{\Delta}{\can(t,\{e_x\}_{x\in X})}[\Inv_{\arr{u}{v}}(t)]
  }{
    \typing{\Delta}{ \leftunit(\can(t,\set{e_x})) \equiv
      t^{\leftunit}}[\arr{t^{\leftinv}\ast t}{\id v}]
  }
  \and
  \inferdef[\(\beta_{\can}\)-\(\rightunit\)]{
    \typing{\Delta}{\can(t,\{e_x\}_{x\in X})}[\Inv_{\arr{u}{v}}(t)]
  }{
    \typing{\Delta}{ \rightunit(\can(t,\set{e_x})) \equiv
      t^{\rightunit}}[\arr{t\ast t^{\rightinv}}{\id u}]
  }
  \and
  \inferdef[\(\beta_{\can}\)-\(\leftunitwitness\)]{
    \typing{\Delta}{\can(t,\{e_x\}_{x\in X})}[\Inv_{\arr{u}{v}}(t)]
  }{
    \typing{\Delta}{
      \leftunitwitness(\can(t,\set{e_x})) \equiv \can(t^{\leftunit}, \{-\})
    }[\Inv(t^{\leftunit})]
  }
  \and
  \inferdef[\(\beta_{\can}\)-\(\rightunitwitness\)]{
    \typing{\Delta}{\can(t,\{e_x\}_{x\in X})}[\Inv_{\arr{u}{v}}(t)]
  }{
    \typing{\Delta}{ \rightunitwitness(\can(t,\{e_x\})) \equiv
      \can(t^{\rightunit},\{-\})}[\Inv(t^{\rightunit})]
  }
\end{mathpar}
Finally to provide the computation rules for the constructor \(\rec\), we define
for every term \(\typing{\Equiv{n+1}}{r}[\Inv(t)]\) a substitution
\(\typing{\Equiv{n+1}}{\instantiation{r}}[\EquivInd{n+1}{t}]\) which
instantiates the new variables of \(\EquivInd{n+1}{t}\) with corecursive calls
to the construction as follows:
\[
  \instantiation{r} = \id_{\Equiv{n+1}}\extsub
  (h^-\mapsto\Susp(r)[\classifierequiv{\leftunitwitness(\evar_{n+1})}]) \extsub
  (h^+\mapsto\Susp(r)[\classifierequiv{\rightunitwitness(\evar_{n+1})}])
\]
Using this substitution, the computation rules for \(\rec\) take the form:
\begin{mathparpagebreakable}
  \inferdef[\(\beta_{\rec}\)-\(\leftinv\)]{
    \typing{\Gamma}{\rec(t,t_l,t_r,t_{lu},t_{ru},t_{ilu},
      t_{iru},\gamma)}[\Inv_{\arr[A]{u}{v}}(t[\gamma])]
  }{
    \typing{\Gamma}{\leftinv(\rec(t,t_l,t_r,t_{lu},t_{ru},
      t_{ilu},t_{iru},\gamma))
      \equiv t_l[\gamma]}[(\arr[A]{v}{u})]
  }
  \and
  \inferdef[\(\beta_{\rec}\)-\(\rightinv\)]{
    \typing{\Gamma}{\rec(t,t_l,t_r,t_{lu},t_{ru},t_{ilu},
      t_{iru},\gamma)}[\Inv_{\arr[A]{u}{v}}(t[\gamma])]
  }{
    \typing{\Gamma}{\rightinv(\rec(t,t_l,t_r,t_{lu},t_{ru},
      t_{ilu},t_{iru},\gamma))
      \equiv t_r[\gamma]}[(\arr[A]{v}{u})]
  }
  \and
  \inferdef[\(\beta_{\rec}\)-\(\leftunit\)]{
    \typing{\Gamma}{\rec(t,t_l,t_r,t_{lu},t_{ru},t_{ilu},
      t_{iru},\gamma)}[\Inv_{\arr[A]{u}{v}}(t[\gamma])]
  }{
    \typing{\Gamma}{\leftunit(\rec(t,t_l,t_r,t_{lu},t_{ru},
      t_{ilu},t_{iru},\gamma))
      \equiv t_{lu}[\gamma]}[\arr{t_{l}[\gamma]\ast t[\gamma]}{\id v}]
  } \and
  \inferdef[\(\beta_{\rec}\)-\(\rightunit\)]{
    \typing{\Gamma}{\rec(t,t_l,t_r,t_{lu},t_{ru},t_{ilu},
      t_{iru},\gamma)}[\Inv_{\arr[A]{u}{v}}(t[\gamma])]
  }{
    \typing{\Gamma}{\rightunit(\rec(t,t_l,t_r,t_{lu},t_{ru},
      t_{ilu},t_{iru},\gamma))
      \equiv t_{ru}[\gamma]}[\arr{t[\gamma]\ast t_r[\gamma]}{\id u}]
  } \and
  \inferdef[\(\beta_{\rec}\)-\(\leftunitwitness\)]{
    \typing{\Gamma}{\rec(t,t_l,t_r,t_{lu},t_{ru},t_{ilu},
      t_{iru},\gamma)}[\Inv_{\arr[A]{u}{v}}(t[\gamma])]
  }{
    \typing{\Gamma}{\leftunitwitness(\rec(t,t_l,t_r,t_{lu},t_{ru},
      t_{ilu},t_{iru},\gamma))} \\ \\
      \qquad\qquad\qquad\qquad\qquad\qquad
      \equiv t_{ilu}[\instantiation{\rec(t,t_l,t_r,t_{lu},t_{ru},
        t_{ilu},t_{iru},\gamma)}\circ\gamma]\ty \Inv(t_{lu}[\gamma])
  } \and
  \inferdef[\(\beta_{\rec}\)-\(\rightunitwitness\)]{
    \typing{\Gamma}{\rec(t,t_l,t_r,t_{lu},t_{ru},t_{ilu},
    t_{iru},\gamma)}[\Inv_{\arr[A]{u}{v}}(t[\gamma])]
  }{
    \typing{\Gamma}{\rightunitwitness(\rec(t,t_l,t_r,t_{lu},t_{ru},
      t_{ilu},t_{iru},\gamma))} \\ \\
      \qquad\qquad\qquad\qquad\qquad\qquad
      \equiv t_{iru}[\instantiation{\rec(t,t_l,t_r,t_{lu},t_{ru},t_{ilu},
        t_{iru},\gamma)}\circ \gamma]\ty \Inv(t_{ru}[\gamma])
  }
\end{mathparpagebreakable}
This concludes the description of the syntax, the typing and the computation
rules of \cattinv.


\section{Proofs in \texorpdfstring{\cattinv}{ICaTT}}
\label{app:proofs-artefact}

This section contains formal proofs of the results in \Cref{sec:examples} that
we have verified in \cattinv. These proofs are also provided as file
\verb|proofs/invertibility.catt| in the supplementary material (available at
\url{https://zenodo.org/records/18343317}), and they have been verified in our
prototype implementation of the type theory \cattinv that can also be found in
the supplementary material.

\compinv*
\begin{proof}
This follows immediately using the constructor \(\can\):
\begin{verbatim}
let compinv (x : *) (y : *) (z : *)(f : x -> y) (g : y -> z)
            (e : Inv (f)) (e' : Inv (g)) : Inv (comp f g)
= can ( comp f g { e , e' })
\end{verbatim}
\vspace{-1.5\baselineskip}
\end{proof}

\invisinv*
\begin{proof}
We first construct a term \(\operatorname{lri}\) that relates the left and
right inverses of a term, illustrated in \Cref{fig:lri}. This is a composite
of five terms and uses the following coherences, corresponding to left and right
whiskerings, an associator, the left unitor and the inverse of the right unitor:
\begin{verbatim}
coh whiskl (x(f)y(g(a)h)z) : comp f g -> comp f h
coh whiskr (x(f(a)g)y(h)z) : comp f h -> comp g h
coh assoc (x(f)y(g)z(h)w) : comp f (comp g h) -> comp (comp f g) h
coh unitl (x(f)y) : comp (id _) f -> f
coh unitr- (x(f)y) : f -> comp f (id _)

let lri (x : *) (y : *) (f : x -> y) (e : Inv(f)) : linv (e) -> rinv (e)
= comp (unitr- (linv (e)))
       (whiskl (linv (e)) (linv (irunit (e))))
       (assoc _ _ _)
       (whiskr (lunit (e)) (rinv (e)) )
       (unitl (rinv(e)))
\end{verbatim}
We then provide an auxiliary term \(\operatorname{lriU-aux}\) that shows that
the term \(\operatorname{lri}\) is invertible, providing that the left inverse
of the right cancellator is invertible. This is again a simple application of
the constructor \(\can\):
\begin{verbatim}
let lriU-aux (x : *) (y : *) (f : x -> y) (e : Inv(f))
             (e' : Inv (linv (irunit (e)))) : Inv (lri e)
= can (lri e {
        can (_ {}) ,
        can (_ { e' }) ,
        can (_ {}) ,
        can (_ { ilunit (e) }) ,
        can (_ {})
})
\end{verbatim}
Using this auxiliary term, we can construct an invertibility structure
for the left inverse of a term, as required, using the corecursor \(\rec\):
\begin{verbatim}
corec linv-inv (x : *) (y : *) (f : x -> y) (e : Inv(f))
= { linv(e) ,
    f ,
    f ,
    comp (whiskl f (lri e)) (runit (e)) ,
    lunit (e) ,
    can (_ {can (_ { lriU-aux IHright }) , (irunit (e))}),
    ilunit (e)
}
\end{verbatim}
Invertibility of the right inverse can be proven in a symmetric way.
\end{proof}

\leftisright*
\begin{proof}
This is an easy consequence of the proof of \Cref{prop:left-inv}. The term
relating the left and right inverses is given by \(\operatorname{lri}\), which
can be shown to be invertible combining \(\operatorname{lriU-aux}\) and
\(\operatorname{linv-inv}\):
\begin{verbatim}
let lriU (x : *) (y : *) (f : x -> y) (e : Inv(f)) : Inv (lri e)
= lriU-aux (linv-inv irunit (e))
\end{verbatim}
\vspace{-1.5\baselineskip}
\end{proof}

\transport*
\begin{proof}
Before proving the result, we complete the proof of \Cref{prop:left-inv} by
providing an explicit invertibility structure \(\operatorname{rinv-inv}\) for
the right inverse of a term as well. Using \Cref{prop:left-inv} and the
constructor \(\coind\), we can define such a term by:
\begin{verbatim}
inv rinv-inv (x : *) (y : *) (f : x -> y) (e : Inv(f))
= { rinv(e) ,
    f ,
    f ,
    runit (e) ,
    comp (whiskr (linv (lriU e)) f) (lunit (e)) ,
    irunit (e) ,
    can (_ {can (_ { linv-inv (lriU e) }) , (ilunit (e))})
}
\end{verbatim}
The proof of the proposition is then again an application of the constructor
\(\coind\):
\begin{verbatim}
inv transport (x : *) (y : *) (f : x -> y) (e : Inv (f))
              (g : x -> y) (a : f -> g) (e' : Inv (a))
= {g ,
   linv (e) ,
   rinv (e) ,
   comp (whiskl (linv (e)) (linv (e'))) (lunit (e)) ,
   comp (whiskr (linv (e')) (rinv (e))) (runit (e)) ,
   can (_ {can (_ { linv-inv e' }) , ilunit (e)}) ,
   can (_ {can (_ { linv-inv e' }) , irunit (e)})
}
\end{verbatim}
\vspace{-1.5\baselineskip}
\end{proof}

\twoofsix*
\begin{proof}
We recall that the 2-of-6 property states that given three composable cells
\(f\), \(g\) and \(h\) such that \(f\ast g\) and \(g\ast h\) are invertible,
then so are \(f\), \(g\) and \(h\). The proof of this property has multiple
parts and uses the following additional coherences: the inverse of the left
unitor and the associator, a ternary unitor and a two-sided whiskering, as well
two associators for five composable cells:
\begin{verbatim}
coh unitl- (x(f)y) : f -> comp (id _) f
coh assoc- (x(f)y(g)z(h)w) : comp (comp f g) h -> comp f (comp g h)
coh unit3 (x(f)y(g)z) : comp f (id _) g -> comp f g
coh whisk3 (x(f)y(g(a)h)z(k)w) : comp f g k -> comp f h k
coh assoc-le (x(f)y(g)z(h)w(k)v(l)u)
: comp (comp (comp f g) h) (comp k l) -> comp f (comp g (comp h k)) l
coh assoc-re (x(f)y(g)z(h)w(k)v(l)u)
: comp (comp f g) (comp h (comp k l)) -> comp f (comp (comp g h) k) l
\end{verbatim}
The proof that \(g\) is invertible can then be constructed directly as follows:
\begin{verbatim}
inv 2of6-g (x : *) (y : *) (z : *) (w : *) (f : x -> y) (g : y -> z)
           (h : z -> w)(e : Inv(comp f g)) (e' : Inv(comp g h))
= { g ,
    comp (linv (e)) f,
    comp h (rinv (e')) ,
    comp (assoc- (linv (e)) f g) (lunit (e)) ,
    comp (assoc g h (rinv (e'))) (runit (e')) ,
    can (_ { can (_ {}) , ilunit (e) } ) ,
    can (_ { can (_ {}) , irunit (e') } )
}
\end{verbatim}
We then embark to prove that \(f\) is invertible, constructing the left and
right cancellators for \(f\) explicitly as well as their witnesses of
invertibility.
\begin{verbatim}
let 2of6-f-runit (x : *) (y : *) (z : *) (w : *) (f : x -> y) (g : y -> z)
                 (h : z -> w) (e : Inv(comp f g)) (e' : Inv(comp g h))
                 : comp f (comp g (rinv (e))) -> id x
= comp (assoc f g rinv(e)) (runit (e))

let 2of6-f-rwit (x : *) (y : *) (z : *) (w : *) (f : x -> y) (g : y -> z)
                (h : z -> w) (e : Inv(comp f g)) (e' : Inv(comp g h))
                : Inv(2of6-f-runit e e')
= can (2of6-f-runit e e' { can (_ {}) , irunit (e)})

let 2of6-f-lunit (x : *) (y : *) (z : *) (w : *) (f : x -> y) (g : y -> z)
                 (h : z -> w) (e : Inv(comp f g)) (e' : Inv(comp g h))
                 : comp (comp g (linv (e))) f -> id y
= comp (unitr- (comp (comp g (linv (e))) f))
       (whiskl _ (linv (irunit (2of6-g e e'))))
       (assoc-le _ _ _ _ _)
       (whisk3 _ (lunit (e)) _)
       (unit3 _ _)
       (runit (2of6-g e e'))

let 2of6-f-lwit (x : *) (y : *) (z : *) (w : *) (f : x -> y) (g : y -> z)
                (h : z -> w) (e : Inv(comp f g)) (e' : Inv(comp g h))
                : Inv (2of6-f-lunit e e')
= can ( 2of6-f-lunit e e' {
        can(_ {}) ,
        can(_ {linv-inv (irunit (2of6-g e e'))}) ,
        can(_ {}) ,
        can(_ { ilunit (e) }) ,
        can(_ {}) ,
        (irunit (2of6-g e e')) })

inv 2of6-f (x : *) (y : *) (z : *) (w : *) (f : x -> y) (g : y -> z)
           (h : z -> w) (e : Inv(comp f g)) (e' : Inv(comp g h))
= { f ,
    comp g (linv (e)) ,
    comp g (rinv (e)) ,
    2of6-f-lunit e e' ,
    2of6-f-runit e e' ,
    2of6-f-lwit e e' ,
    2of6-f-rwit e e'
}
\end{verbatim}
Finally, we can construct the left and right cancellators for \(h\) and their
witnesses of invertibility:
\begin{verbatim}
let 2of6-h-lunit (x : *) (y : *) (z : *) (w : *) (f : x -> y) (g : y -> z)
                 (h : z -> w) (e : Inv(comp f g)) (e' : Inv(comp g h))
                 : comp (comp (linv (e')) g) h -> id w
= comp (assoc- (linv (e')) g h) (lunit (e'))

let 2of6-h-lwit (x : *) (y : *) (z : *) (w : *) (f : x -> y) (g : y -> z)
                (h : z -> w) (e : Inv(comp f g)) (e' : Inv(comp g h))
                : Inv (2of6-h-lunit e e')
= can (2of6-h-lunit e e' { can(_ {}) , ilunit (e')})

let 2of6-h-runit (x : *) (y : *) (z : *) (w : *) (f : x -> y) (g : y -> z)
                 (h : z -> w) (e : Inv(comp f g)) (e' : Inv(comp g h))
                 : comp h (comp (rinv (e')) g) -> id z
= comp (unitl- (comp h (comp (rinv (e')) g)))
       (whiskr (rinv (ilunit (2of6-g e e'))) _)
       (assoc-re (linv (2of6-g e e')) g _ _ _)
       (whisk3 _ (runit (e')) _)
       (unit3 _ _)
       (lunit (2of6-g e e'))

let 2of6-h-rwit (x : *) (y : *) (z : *) (w : *) (f : x -> y) (g : y -> z)
                (h : z -> w) (e : Inv(comp f g)) (e' : Inv(comp g h))
                : Inv(2of6-h-runit e e')
= can ( 2of6-h-runit e e' {
        can(_ {}) ,
        can(_ {rinv-inv (ilunit (2of6-g e e'))}) ,
        can(_ {}) ,
        can(_ { irunit (e') }) ,
        can(_ {}) ,
        (ilunit (2of6-g e e')) })

inv 2of6-h (x : *) (y : *) (z : *) (w : *) (f : x -> y) (g : y -> z)
           (h : z -> w) (e : Inv(comp f g)) (e' : Inv(comp g h))
= { h ,
    comp (linv (e')) g ,
    comp (rinv (e')) g ,
    2of6-h-lunit e e' ,
    2of6-h-runit e e' ,
    2of6-h-lwit e e' ,
    2of6-h-rwit e e'
}
\end{verbatim}
\vspace{-1.5\baselineskip}
\end{proof}


\end{document}